\documentclass[10pt,letterpaper]{article}
\usepackage[utf8]{inputenc}
\usepackage{amsmath}
\usepackage{amsfonts}
\usepackage{amssymb}
\usepackage{graphicx}
\usepackage{verbatim}
\usepackage{mathrsfs}
\usepackage{upref,amsthm,amsxtra,exscale}
\usepackage{cite}
\usepackage{dsfont}
\usepackage[colorlinks=true,urlcolor=blue,
citecolor=red,linkcolor=blue,linktocpage,pdfpagelabels,
bookmarksnumbered,bookmarksopen]{hyperref}
\usepackage{textcomp}
\usepackage{subcaption}

\usepackage{marvosym}

\usepackage{fullpage}

\newtheorem{theorem}{Theorem}[section]
\newtheorem{corollary}[theorem]{Corollary}
\newtheorem{remark}[theorem]{Remark}

\newtheorem{lemma}[theorem]{Lemma}
\newtheorem{proposition}[theorem]{Proposition}
\newtheorem{definition}[theorem]{Definition}
\newtheorem{examples}[theorem]{Examples}

\newtheorem{question}[theorem]{Question}

\numberwithin{equation}{section}

\def\R{\mathbb{R}}
\def\r{\mathbb{R}}
\def\rn{\mathbb{R}^N}
\def\z{\mathbb{Z}}
\def\z2{\mathbb{Z}_2}

\def\s1{\mathbb{S}^1}
\def\n{\mathbb{N}}

\def\eps{\varepsilon}
\def\rh{\rightharpoonup}
\def\io{\int_{\Omega}}
\def\irn{\int_{\r^N}}
\def\vp{\varphi}
\def\vt{\vartheta}
\def\vr{\varrho}
\def\o{\Omega}
\def\t{\Theta}

\def\cC{\mathcal{C}}

\def\cN{\mathcal{N}}

\def\cU{\mathcal{U}}

\def\supp{\mathrm{supp}}
\def\what{\widehat}
\def\tilde{\widetilde}
\def\d{\,\mathrm{d}}

\def\dist{\mathrm{dist}}
\def\dim{\mathrm{dim}}
\def\cat{\mathrm{cat}}
\def\cupl{\mathrm{cup\text{-}length}}

\author{Mónica Clapp, \ Jorge Faya\footnote{J. Faya is supported  by ANID (Chile) through the project FONDECYT/Iniciacion 11190423}  \ and \ Alberto Saldaña\footnote{A. Saldaña is supported by UNAM-DGAPA-PAPIIT (Mexico) grant IN102925  and by SECIHTI (Mexico) grant CBF2023-2024-116.}\ \thanks{\Letter\ Corresponding author: \texttt{alberto.saldana@im.unam.mx}}
}

\title{Critical equations with a sharp change of sign in the nonlinearity}
\date{}

\begin{document}
\maketitle

\begin{abstract}
We establish the existence and nonexistence of entire solutions to a semilinear elliptic problem whose nonlinearity is the critical power multiplied by a function that takes the value $1$ in an open bounded region of $\rn$ and the value $-1$ in its complement. The existence or not of solutions depends on the geometry of the bounded region, in a way analogous to what happens with the classical critical Dirichlet problem in a bounded domain. Our methods are variational and include the use of topological tools.
\smallskip

\emph{Key words and phrases:} Sign-changing nonlinearity, critical exponent, entire solutions.
\smallskip

\emph{2020 Mathematics Subject Classification:} 35J61, 35B33, 35B08, 35J20.
\end{abstract}
	
\section{Introduction}
	
Let $\o$ be a bounded smooth open subset of $\rn$, not necessarily connected, $N\geq 3$, and $Q_\o:\rn\to\r$ be given by
\begin{equation*}
Q_\o(x):=
\begin{cases}
1 & \text{if \ }x\in\o, \\
-1 & \text{if \ }x\in\rn\smallsetminus\o.
\end{cases}
\end{equation*}
Consider the critical problem
\begin{equation} \label{eq:problem}
\begin{cases}
-\Delta u + \lambda\mathds{1}_\o u = Q_\o(x)|u|^{2^*-2}u,\\
u\in D^{1,2}(\rn),
\end{cases} 
\end{equation}
where $\mathds{1}_\o$ is the characteristic function of $\o$, $2^*:=\frac{2N}{N-2}$ is the critical Sobolev exponent, and $\lambda\in(-\Lambda_\o,0]$ with
\begin{equation}\label{eq:Lambda}
\Lambda_{\Omega}:=\inf_{\substack{u\in D^{1,2}(\mathbb{R}^N)\\ u\neq 0}}\frac{\int_{\mathbb{R}^{N}}|\nabla u|^{2}}{\int_{\Omega}u^{2}}.
\end{equation}

This problem arises in some nonlinear optics models describing the behavior of optical waveguides propagating through a stratified dielectric medium, such as those described in \cite{st,st2}. In \cite{as}, Ackermann and Szulkin analyzed the behavior of ground-state solutions for such models as the region where self-focusing occurs shrinks, and showed that a concentration phenomenon occurs. Later, in \cite{fw}, Fang and Wang described their asymptotic behavior by identifying their limiting profile as a minimum energy solution of \eqref{eq:problem}. This limiting equation, with both subcritical and critical nonlinearities, has been further studied, for instance, in \cite{chs,cps,l}, and similar elliptic systems were considered in \cite{css,jsx,zz}.

Although the solutions of problem \eqref{eq:problem} are defined in all $\rn$, we will show in this work that their existence or nonexistence depends solely on the open set $\o$ and the number $\lambda$, in a manner analogous to what happens with the critical equation
\begin{equation} \label{eq:classical}
-\Delta u + \lambda u = |u|^{2^*-2}u,\qquad u\in D_0^{1,2}(\o),
\end{equation}
in a bounded domain $\o\subset\rn$.

A  classical result of Brezis and Nirenberg \cite{bn} establishes the existence of a minimum energy solution to \eqref{eq:classical}
in any bounded domain $\o\subset\rn$, $N\geq 4$, for any $\lambda\in(-\lambda_1(\o),0)$, where $\lambda_1(\o)$ is the first Dirichlet eigenvalue of $-\Delta$ in $\o$. Our first result says that problem \eqref{eq:problem} behaves similarly.

\begin{theorem}\label{thm:bn}
If $N\geq 4$ and $\lambda\in(-\Lambda_\o,0),$ then the problem \eqref{eq:problem} has a positive least energy solution.
\end{theorem}

For $\lambda=0$ the existence or nonexistence of solutions to the classical problem \eqref{eq:classical} depends on the geometry of the domain. It is well known that \eqref{eq:classical} has no minimum energy solution in a bounded domain $\o$, and that it does not have a nontrivial solution when $\o$ is strictly starshaped; see, for instance, \cite[Chapter III]{s}. We show that the same is true for problem \eqref{eq:problem}.

\begin{theorem}\label{thm:nonexistence}
Let $\lambda=0$. Then, the following statements hold true.
\begin{itemize}
\item[$(i)$] The problem \eqref{eq:problem} does not have a least energy solution.
\item[$(ii)$] If $\o$ is strictly starshaped, then \eqref{eq:problem} admits only the trivial solution.
\end{itemize}
\end{theorem}

There are several existence results for the classical problem \eqref{eq:classical} with $\lambda=0$ in bounded domains with nontrivial topology. One of the first is the result of Coron \cite{c} who showed that in a bounded domain with a sufficiently small point puncture, this problem is solvable.
Our next result considers more general holes, obtained by removing a sufficiently thin neighborhood of a closed submanifold. It asserts the existence of multiple solutions, whose number depends on the topology of the manifold. To state this result we recall the classical notion of cup-length of a topological space and we introduce some notation.

Let $\tilde H^*(\,\cdot\,;\z2)$ be reduced singular cohomology with $\z2$-coefficients. The \emph{cup-length of $X$} is the smallest number $m$ such that the cup product of any $m$ cohomology classes in $\tilde H^*(X;\z2)$ is equal to zero. If $X$ is a subset of $\rn$ and $0<r\leq R$ we write
\begin{align*}
\bar B_R X:=\{x\in\rn:\dist(x,X)\leq R\}\qquad\text{and}\qquad A_{r,R}X:=\{x\in\rn:r\leq \dist(x,X)\leq R\}.
\end{align*}
Let $S$ be the best Sobolev constant for the embedding of $D^{1,2}(\rn)$ into $L^{2^*}(\rn)$.

\begin{theorem}\label{thm:coron}
Let $M$ be a closed (i.e., compact and without boundary) smooth submanifold of $\rn$ and $\bar B_RM$ be a closed tubular neighborhood of $M$. Then, for any given $b\in(\frac{1}{N}S^\frac{N}{2},\frac{2}{N}S^\frac{N}{2})$, there exists $r\in(0,R)$ such that, if $\o$ satisfies
\begin{equation}\label{eq:assumptions}
M\cap\overline{\o}=\emptyset\qquad\text{and}\qquad A_{r,R}M\subset\o,
\end{equation}
then the problem \eqref{eq:problem} with $\lambda=0$ has at least 
$$\cupl(M)$$
positive solutions whose energy is in $(\frac{1}{N}S^\frac{N}{2},b]$. In particular, \eqref{eq:problem} has at least one positive solution. 
\end{theorem}

Let us give some examples.

\begin{examples} \label{examples_coron}
Let $\lambda=0$.
\begin{enumerate}
\item[$(a)$] If $M$ is a point, the assumptions \eqref{eq:assumptions} coincide with those given by Coron in \emph{\cite[Theorem 1]{c}} for the classical problem \eqref{eq:classical}. The reduced cohomology of a point is trivial, so $\cupl(M)=1$ in this case; see \emph{Figure \ref{fig:sub1}}.
\item[$(b)$] If $M$ is homeomorphic to the unit circle $\s1\subset \r^2$, its reduced cohomology is $\tilde H^j(M;\z2)=\z2$ if $j=1$ and $\tilde H^j(M;\z2)=0$ otherwise, so $\cupl(M)=2$ and \emph{Theorem \ref{thm:coron}} yields at least $2$ positive solutions to problem \eqref{eq:problem}; see \emph{Figure \ref{fig:sub2}}.
\item[$(c)$] If $M$ is homeomorphic to $\s1\times\s1$, its reduced cohomology is $\tilde H^1(M;\z2)=\z2\oplus\z2$, $\tilde H^2(M;\z2)=\z2$ and $\tilde H^j(M;\z2)=0$ if $j\neq 1,2$. The product of the two generators of $\tilde H^1(M;\z2)$ is the generator of $\tilde H^2(M;\z2)$, so $\cupl(M)=3$ and \emph{Theorem \ref{thm:coron}} yields at least $3$ positive solutions to problem \eqref{eq:problem}; see \emph{Figure \ref{fig:sub3}}.
\item[$(d)$] More generally, if $M$ is homeomorphic to the unit sphere $\mathbb{S}^n\subset\r^{n+1}$, its reduced cohomology is $\tilde H^j(M;\z2)=\z2$ if $j=n$ and $\tilde H^j(M;\z2)=0$ otherwise, so $\cupl(M)=2$ and \emph{Theorem \ref{thm:coron}} yields at least $2$ positive solutions to problem \eqref{eq:problem}. If $M$ is homeomorphic to a product of $m$ spheres $\mathbb{S}^{n_1}\times\cdots\times\mathbb{S}^{n_m}$, then $\cupl(M)=m+1$ and problem \eqref{eq:problem} has at least $m+1$ positive solutions. See \emph{Figure \ref{fig:sub4}}.
\end{enumerate}
\end{examples}

\begin{figure}[htbp]
    \centering
    \begin{subfigure}[b]{0.3\textwidth}
        \centering
        \includegraphics[width=0.8\textwidth]{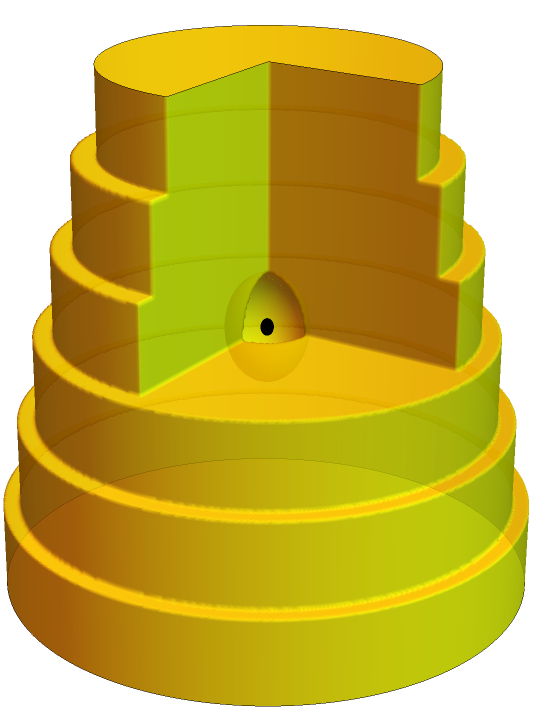}
        \caption{$M=\{0\}$.}
        \label{fig:sub1}
    \end{subfigure}
    \begin{subfigure}[b]{0.6\textwidth}
        \centering
        \includegraphics[width=0.45\textwidth]{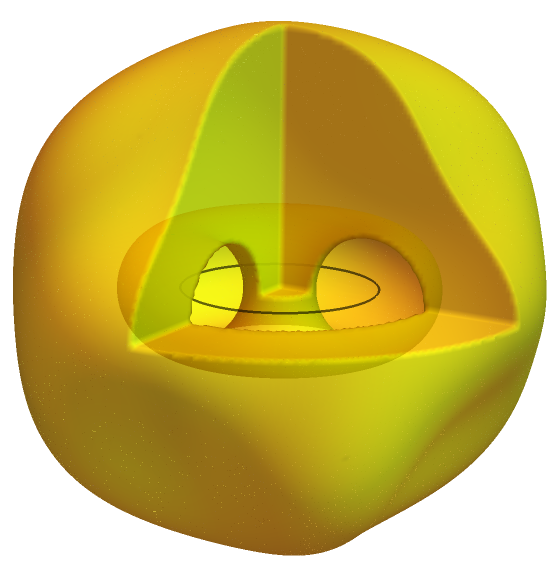}
        \includegraphics[width=0.45\textwidth]{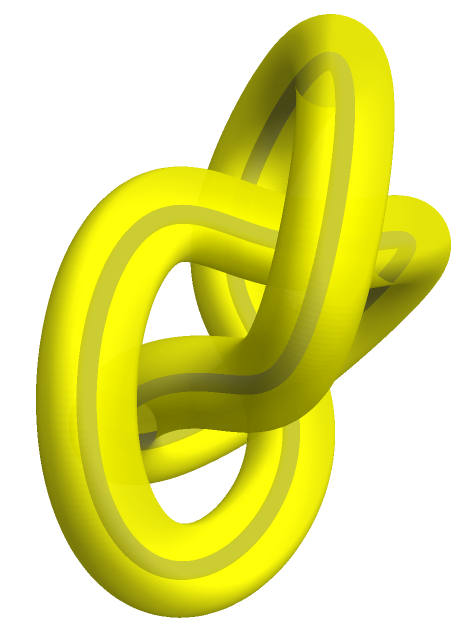}
        \caption{$M\approx\s1$. A simple closed curve is homeomorphic to $\s1$.}
        \label{fig:sub2}
    \end{subfigure}
    \begin{subfigure}[b]{0.4\textwidth}
        \centering
        \includegraphics[width=\textwidth]{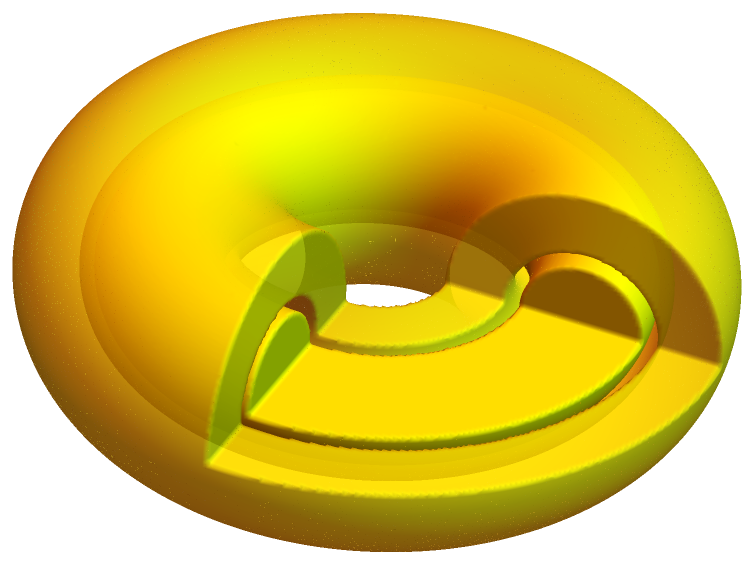}
        \caption{$M\approx\mathbb{S}^1\times \mathbb{S}^1$.}
        \label{fig:sub3}
    \end{subfigure}
    \qquad
    \begin{subfigure}[b]{0.33\textwidth}
        \centering
        \includegraphics[width=\textwidth]{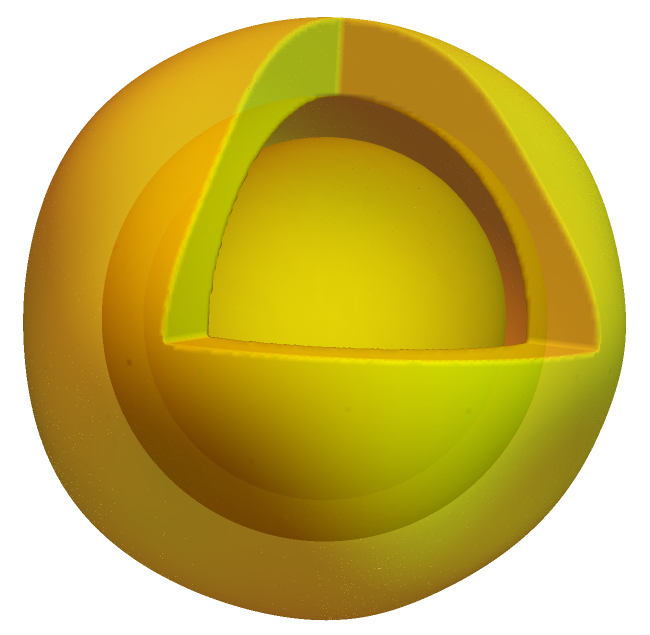}
        \caption{$M\approx\mathbb{S}^2$.}
        \label{fig:sub4}
    \end{subfigure}
    \caption{
     Some examples of three-dimensional domains satisfying \eqref{eq:assumptions} for different manifolds $M$. A sectional cut is presented to see the inner holes.}
    \label{figs}
\end{figure}
\medskip

As in the classical case \eqref{eq:classical}, symmetries influence the number of solutions to the problem \eqref{eq:problem}. Several symmetric results that hold for \eqref{eq:classical} have an analogous version for \eqref{eq:problem}. Our next result illustrates this fact.

Let \( O(N) \) be the group of linear isometries of \( \mathbb{R}^N \). If \( G \) is a closed subgroup of \( O(N) \), we define the \( G \)-orbit of a point \( x \in \mathbb{R}^N \) as
\[
Gx := \{ gx : g \in G \},
\]
and denote its cardinality by \( \#Gx \). Recall that \( \Omega \subset \mathbb{R}^N \) is called \( G \)-invariant if \( Gx \subset \Omega \) for all \( x \in \Omega \), and a function \( u : \mathbb{R}^N \to \mathbb{R} \) is \( G \)-invariant if it is constant on every \( G \)-orbit \( Gx \).

\begin{theorem}\label{thm:sym1}
Let $\lambda\in (-\Lambda_\Omega,0].$ If $\Omega$ is $G$-invariant and if $\#Gx = \infty$ for all $x \in \overline{\o}$, then the problem \eqref{eq:problem} has infinitely many $G$-invariant solutions.
\end{theorem}

Note that \emph{only the $G$-orbits in $\overline{\o}$} are required to be infinite.

\begin{corollary} \label{cor:annuli}
If $\o$ is a finite disjoint union of annuli, i.e.,
$$\o:=\bigcup_{i=1}^n\{x\in\rn:a_i<|x|<b_i\},\qquad 0<a_1<b_1<a_2<\cdots<a_n<b_n$$
for some $n\in\mathbb N,$ then \eqref{eq:problem} has infinitely many radial solutions.
\end{corollary}

This contrasts with the behavior of the problem
$$-\Delta u=|u|^{2^*-2}u,\qquad u\in D^{1,2}(\rn),$$
which has only one radial solution.

It is not enough that $\o$ be radial for Corollary \ref{cor:annuli} to be true. The statement is certainly not true if $\o$ is a ball, as Theorem \ref{thm:nonexistence} shows. But what if $\o$ is the disjoint union of a ball, centered at the origin, and an annulus? Does \eqref{eq:problem} have a solution? If the gap between the ball and the annulus is thin enough, the answer is yes, as shown in Example \ref{examples_coron}$(d)$. The general case seems to be much harder. We leave it as an open question.

\begin{question}
Assume that $\o$ is the disjoint union of a ball, centered at the origin, and an annulus. Does problem \eqref{eq:problem} have a solution?
\end{question}

To show that problem \eqref{eq:problem} has no nontrivial solution when $\o$ is strictly starshaped, we derive a Poho\v{z}aev identity, that holds true for a more general problem; see Theorem \ref{thm:pohozhaev}.

On the other hand, the proof of the existence results relies on a delicate analysis of the loss of compactness of the variational functional associated with the problem \eqref{eq:problem}. We show that blow-up can occur only at points in $\o$. This fact is responsible for the similarities in the behavior of \eqref{eq:problem} with that of the classical critical equation \eqref{eq:classical} in a bounded domain; see Theorem \ref{thm:struwe}.

As a consequence of this analysis, and using the well-known Brezis-Nirenberg estimates, we obtain Theorem \ref{thm:bn} by minimizing the energy. Theorem \ref{thm:struwe} also allows us to conclude that symmetries can help to restore compactness and obtain multiplicity of solutions results, such as the one proposed in Theorem \ref{thm:sym1}.

Theorem \ref{thm:coron} is more delicate because, as Theorem \ref{thm:nonexistence} shows, the minimal energy is not attained if $\lambda=0$. However, Theorem \ref{thm:struwe} gives an energy interval in which compactness holds, and we use the thin enough $M$-shaped hole in the domain $\o$ to produce nontrivial topology in an interval $(a,b)$ within those levels. More precisely, using the fact that blow-up can occur only at points in $\o$, we are able to construct maps of pairs
$$(\bar B_RM,\partial(\bar B_RM))\to (\cN_0^{\leq b},\cN_0^{\leq a})\to(\rn,\rn\smallsetminus M)$$
whose composition is the inclusion, where $\cN_0^{\leq c}$ is the sublevel set for $c$ in the Nehari manifold $\cN_0$. As a consequence, the number of critical points of the variational functional is bounded from below by the Lusternik-Schnirelmann category of the inclusion of pairs; see Theorem \ref{thm:coron_ls}. Then, using Thom's isomorphism, we show that the category of this inclusion is bounded from below by the cup-length of $M$.

When $M$ is a single point, the profile and blow-up rate of the solution as the hole shrinks are described in \cite{cps}. There, solutions are also constructed in the slightly subcritical case using the Lyapunov-Schmidt reduction procedure. A result related to Theorem \ref{thm:coron} for the classical problem \eqref{eq:classical} is given in \cite{cgp}.

The paper is organized as follows. The variational setting is described in Section \ref{sec:setting}. In Section \ref{sec:pohozhaev} we derive a Poho\v{z}aev identity for problem \eqref{eq:problem} and we prove Theorem \ref{thm:nonexistence}. In Section \ref{sec:compactness} we analyze the loss of compactness and we prove Theorem \ref{thm:bn}. Section \ref{sec:coron} is devoted to the proof of Theorem \ref{thm:coron} and Section \ref{sec:symmetries} to that of Theorem \ref{thm:sym1}.

\section{The variational problem}
\label{sec:setting}
	
We start by explaining the relevance of the number $\Lambda_\o$ defined in \eqref{eq:Lambda}. Note first that it satisfies the following inequalities.

\begin{lemma} \label{lem:Lambda}
$0<\Lambda_\o\leq\lambda_1(\o)$, where $\lambda_1(\o)$ is the first Dirichlet eigenvalue of $-\Delta$ in $\o$.
\end{lemma}

\begin{proof}
By Hölder's and Sobolev's inequalities, there is a constant $C>0$, depending only on $\o$ and $N$, such that
\begin{align*}
\io u^2\leq |\o|^\frac{2}{N}\left(\io |u|^{2^*}\right)^\frac{2}{2^*}\leq|\o|^\frac{2}{N}\left(\irn |u|^{2^*}\right)^\frac{2}{2^*}\leq C\irn|\nabla u|^2
\end{align*}
for every $u\in D^{1,2}(\rn)$. Therefore, $0<\Lambda_\o$. On the other hand, as $D_0^{1,2}(\o)\subset D^{1,2}(\rn)$ we have that
$$\Lambda_\o=\inf_{\substack{u\in D^{1,2}(\rn)\\ u\neq 0}}\frac{\irn|\nabla u|^2}{\io u^2}\leq \inf_{\substack{u\in D_0^{1,2}(\o)\\ u\neq 0}}\frac{\io|\nabla u|^2}{\io u^2}=\lambda_1(\o),$$
as claimed.
\end{proof}

The main property of $\Lambda_\o$ is the following one.

\begin{lemma}\label{lem:equivalent norm}
For every $\lambda\in(-\Lambda_\o,0]$, the expression
$$\|u\|_\lambda:= \left(\irn(|\nabla u|^2 + \lambda\mathds 1_\o u^2)\right)^{1/2}.$$
defines a norm in $D^{1,2}(\rn)$ that is equivalent to the standard one, denoted by
$$\|u\|:= \left(\irn|\nabla u|^2\right)^{1/2}.$$
\end{lemma}

\begin{proof}
As $\lambda\in(-\Lambda_\o,0]$, for every $u\in D^{1,2}(\rn)$ we have that
$$\|u\|^2=\|u\|^2_\lambda - \lambda\io u^2 \leq \|u\|^2_\lambda - \frac{\lambda}{\Lambda_\o}\|u\|^2.$$
Therefore,
$$\Big(1+\frac{\lambda}{\Lambda_\o}\Big)\|u\|^2\leq\|u\|^2_\lambda\leq\|u\|^2,$$
and $1+\frac{\lambda}{\Lambda_\o}>0$. The statement follows easily.
\end{proof}
	
The solutions to \eqref{eq:problem} are the critical points of the functional $J_\lambda:D^{1,2}(\rn)\to\r$ given by
$$J_\lambda(u):=\frac{1}{2}\|u\|_\lambda^2 - \frac{1}{2^*}\irn Q_\o(x)|u|^{2^*},$$
which is of class $\cC^2$. Its derivative at $u$ is
$$J'_\lambda(u)v=\irn\nabla u\cdot\nabla v + \irn\lambda\mathds 1_\o  uv - \irn Q_\o(x)|u|^{2^*-2}uv.$$
The nontrivial critical points of $J_\lambda$ belong to the Nehari manifold
$$\cN_\lambda:=\{u\in D^{1,2}(\rn):u\neq 0, \ J'_\lambda(u)u=0\}=\{u\in D^{1,2}(\rn):u\neq 0, \ \|u\|_\lambda^2 = \irn Q_\o(x)|u|^{2^*}\},$$
which is a Hilbert submanifold of $D^{1,2}(\rn)$ of class $\cC^2$ and a natural constraint for $J_\lambda$. Note that
$$J_\lambda(u)=\frac{1}{N}\|u\|_\lambda^2=\frac{1}{N}\irn Q_\o(x)|u|^{2^*}\qquad\text{if \ }u\in\cN_\lambda.$$
Set 
$$c_\lambda:=\inf_{u\in\cN_\lambda}J_\lambda(u).$$
A function $u\in\cN_\lambda$ that satisfies $J_\lambda(u)=c_\lambda$ is called a \emph{least energy solution} to the problem \eqref{eq:problem}.

Not every element of $D^{1,2}(\rn)$ admits a radial projection onto $\cN_\lambda$. Those that do, belong to the set
\begin{equation*}
\cU:=\Big\{u\in D^{1,2}(\rn):\irn Q_\o(x)|u|^{2^*}>0\Big\}. 
\end{equation*}
Indeed, for any $u\in\cU$ there exists a unique $\tau_{\lambda,u}\in(0,\infty)$ such that $\tau_{\lambda,u}u\in\cN_\lambda$. Explicitly,
\begin{equation} \label{eq:U}
\tau_{\lambda,u}=\left(\frac{\|u\|_\lambda^2}{\irn Q_\o(x)|u|^{2^*}}\right)^\frac{1}{2^*-2}\qquad\text{and}\qquad J_\lambda(\tau_{\lambda,u}u)=\frac{1}{N}\left(\frac{\|u\|_\lambda^2}{\Big(\irn Q_\o(x)|u|^{2^*}\Big)^{2/2^*}}\right)^\frac{N}{2}.
\end{equation}
Note that $\cN_\lambda\subset\cU$. Therefore,
$$c_\lambda=\frac{1}{N}\left(\inf_{u\in\cU}\frac{\|u\|_\lambda^2}{\Big(\irn Q_\o(x)|u|^{2^*}\Big)^{2/2^*}}\right)^\frac{N}{2}.$$
Let $S$ be the best constant for the Sobolev embedding $D^{1,2}(\rn)\hookrightarrow L^{2^*}(\rn)$.
	
\begin{proposition}\label{prop:minimal energy}
\begin{itemize}
\item[$(a)$] $c_\lambda > 0$ for every $\lambda\in(-\Lambda_\o,0]$.
\item[$(b)$] $c_\lambda < \frac{1}{N}S^\frac{N}{2}$ if $N\geq 4$ and $\lambda\in(-\Lambda_\o,0)$.
\item[$(c)$] $c_0=\frac{1}{N}S^\frac{N}{2}$ and $c_0$ is not attained, i.e., problem \eqref{eq:problem} does not have a least energy solution if $\lambda=0$.
\end{itemize}
\end{proposition}
	
\begin{proof}
$(a):$ \ Let $\lambda\in(-\Lambda_\o,0]$. By Lemma \ref{lem:equivalent norm} and Sobolev's inequality there is $C>0$ such that
$$\|u\|_\lambda^2=\irn Q_\o|u|^{2^*}\leq \irn |u|^{2^*}\leq C\|u\|_\lambda^{2^*}\qquad\text{for every \ }u\in\cN_\lambda.$$
Therefore, $c_\lambda > 0$.

$(b):$ \ Assume that $N\geq 4$ and $\lambda\in(-\Lambda_\o,0)$. Since \ $D^{1,2}_0(\o)\smallsetminus\{0\}\subset\cU$ \ and \ $\Lambda_\o\leq\lambda_1(\o)$ \ we have that 
$$\inf_{u\in \cU}\frac{\|u\|^2}{\Big(\irn Q_\o(x)|u|^{2^*}\Big)^{2/2^*}}\leq \inf_{\substack{u\in D^{1,2}_0(\o) \\ u\neq 0}}\frac{\io(|\nabla u|^2 + \lambda u^2)}{\Big(\io|u|^{2^*}\Big)^{2/2^*}}<S.$$
The last (strict) inequality follows from the classical Brezis-Nirenberg estimates \cite[Lemma 1.1]{bn}. Hence, $c_\lambda < \frac{1}{N}S^\frac{N}{2}$.

$(c):$ \ Note that $D^{1,2}_0(\o)\smallsetminus\{0\}\subset\cU\subset D^{1,2}(\rn)\smallsetminus\{0\}$. Therefore,
$$S=\inf_{\substack{u\in D^{1,2}_0(\o) \\ u\neq 0}}\frac{\io|\nabla u|^2}{\Big(\io|u|^{2^*}\Big)^{2/2^*}}\geq\inf_{u\in \cU}\frac{\|u\|^2}{\Big(\irn Q_\o(x)|u|^{2^*}\Big)^{2/2^*}}\geq\inf_{\substack{u\in D^{1,2}(\rn) \\ u\neq 0}}\frac{\|u\|^2}{\Big(\irn |u|^{2^*}\Big)^{2/2^*}}=:S.$$
The first equality is well known \cite[Lemma III.1.1]{s}. It follows that all inequalities are equalities and
$$\inf_{u\in \cU}\frac{\|u\|^2}{\Big(\irn Q_\o(x)|u|^{2^*}\Big)^{2/2^*}}=S.$$
Hence, $c_0=\frac{1}{N}S^\frac{N}{2}$. If this infimum were attained at some $u\in\cN_0$, from the identity
$$\frac{\|u\|^2}{\Big(\irn Q_\o(x)|u|^{2^*}\Big)^{2/2^*}}=\frac{\|u\|^2}{\Big(\irn |u|^{2^*}\Big)^{2/2^*}}=S$$
we would get that $u\equiv 0$ in $\rn\smallsetminus\o$, contradicting the maximum principle.
\end{proof}

\section{A Poho\v{z}aev identity and a nonexistence result}
\label{sec:pohozhaev}
	
In this section we assume that $\t$ is a \emph{possibly unbounded} smooth open subset of $\rn$ and we set
\begin{equation} \label{eq:Q_t}
Q_\t(x):=
\begin{cases}
1 & \text{if \ }x\in\t, \\
-1 & \text{if \ }x\in\rn\smallsetminus\t.
\end{cases}
\end{equation}
 Let $f:\r\to\r$ be a continuous function and $F:\r\to \r$ be given by
\begin{align*}
F(w):=\int_0^w f(t)\d t.
\end{align*}
We prove the following Poho\v{z}aev identity.
	
\begin{theorem}\label{thm:pohozhaev}
Let $u\in \cC^{1}(\rn)\cap \cC^2(\t)\cap \cC^2(\rn\smallsetminus\overline{\t})\cap D^{1,2}(\rn)$ be a solution to the equation
\begin{equation} \label{eq:problem_pohozhaev}
-\Delta u=Q_\t(x)f(u)
\end{equation}
such that $F(u)\in L^1(\rn)$. Then,
\begin{align*}
\frac{1}{2^*} \int_{\R^N}|\nabla u|^2-\int_{\R^N}Q_\t(x)F(u)\d x  + \frac{2}{N}\int_{\partial\t}F(u) \zeta\cdot \nu_\t \d \zeta=0,
\end{align*}
where $\nu_\t$ denotes the exterior unit normal to $\t$.
\end{theorem}

\begin{proof}
We begin by following the standard argument used to show the usual Poho\v{z}aev identity. Let $u$ be a solution of \eqref{eq:problem_pohozhaev}. Observe that
\begin{align*}
\operatorname{div}(F(u) x) & =f(u)(\nabla u \cdot x)+N F(u), \\
\operatorname{div}((\nabla u \cdot x) \nabla u) & =\nabla(\nabla u \cdot x) \cdot \nabla u+(\nabla u \cdot x) \Delta u,\\
\nabla(\nabla u \cdot x) \cdot \nabla u & =\sum_{i=1}^N \frac{\partial}{\partial x_i}\left(\sum_{j=1}^N \frac{\partial u}{\partial x_j} x_j\right) \frac{\partial u}{\partial x_i}=\sum_{i, j=1}^N \frac{\partial^2 u}{\partial x_i \partial x_j} x_j \frac{\partial u}{\partial x_i}+\sum_{i=1}^N\left|\frac{\partial u}{\partial x_i}\right|^2 \\
& =\sum_{j=1}^N \frac{\partial}{\partial x_j}\left(\frac{|\nabla u|^2}{2}\right) x_j+|\nabla u|^2=\nabla\left(\frac{|\nabla u|^2}{2}\right) \cdot x+|\nabla u|^2.
\end{align*}
Therefore,
\begin{align*}
\operatorname{div}((\nabla u \cdot x) \nabla u) & =\nabla\left(\frac{|\nabla u|^2}{2}\right) \cdot x+|\nabla u|^2+(\nabla u \cdot x) \Delta u \\
& =\operatorname{div}\left(\frac{|\nabla u|^2}{2} x\right)-\frac{N}{2}|\nabla u|^2+|\nabla u|^2+(\nabla u \cdot x) \Delta u\\
& =\operatorname{div}\left(\frac{|\nabla u|^2}{2} x\right)-\frac{N-2}{2}|\nabla u|^2+(\nabla u \cdot x) \Delta u
\end{align*}
Since $u$ solves \eqref{eq:problem_pohozhaev}, we have that
\begin{align}\label{1eq}
0 & =(\Delta u+f(u))(\nabla u \cdot x) =\frac{N-2}{2}|\nabla u|^2-N F(u)+\operatorname{div}\left((\nabla u \cdot x) \nabla u-\frac{|\nabla u|^2}{2} x+F(u) x\right)\quad \text{ in }\t
\end{align}
and
\begin{align}\label{2eq}
0 & =\frac{N-2}{2}|\nabla u|^2+N F(u)+\operatorname{div}\left((\nabla u \cdot x) \nabla u-\frac{|\nabla u|^2}{2} x-F(u) x\right)\quad \text{ in }\mathbb R^N\smallsetminus\overline{\t}.
\end{align}
Let $B_R$ be the ball of radius $R$ in $\rn$, centered at the origin. We integrate \eqref{2eq} over $B_R\smallsetminus\t$ to obtain
\begin{align}\label{3eq}
0&=\frac{N-2}{2} \int_{B_R\smallsetminus\t}|\nabla u|^2
+N \int_{B_R\smallsetminus\t} F(u)+\int_{\partial \t\cup \partial B_R}\left((\nabla u \cdot \zeta) \nabla u-\frac{|\nabla u|^2}{2} \zeta	-F(u) \zeta\right) \cdot \nu_{B_R\smallsetminus\t}\, \d \zeta,
\end{align}
where $\nu_{B_R\smallsetminus\t}$ is the exterior unit normal to $B_R\smallsetminus\t$.  Since $F(u)\in L^1(\R^N)$, we have that
\begin{align*}
\lim_{n\to \infty}\int_{\partial B_{R_n}(0)}\left((\nabla u \cdot \zeta) \nabla u-\frac{|\nabla u|^2}{2} \zeta+F(u) \zeta\right) \cdot \nu_{B_{R_n}} \d \zeta=0
\end{align*}
for some sequence $R_n\to\infty.$ Then, letting $R=R_n\to\infty$ in \eqref{3eq}, we get
\begin{align}\label{4eq}
0=\frac{N-2}{2} \int_{\R^N\smallsetminus\t}|\nabla u|^2+N \int_{\R^N\smallsetminus\t} F(u)	+\int_{\partial \t}\left((\nabla u \cdot \zeta) \nabla u-\frac{|\nabla u|^2}{2} \zeta
-F(u) \zeta\right) \cdot \nu_{\rn\smallsetminus\t} \d \zeta.
\end{align}
Similarly, from \eqref{1eq} we obtain 
\begin{align}
0= \frac{N-2}{2} \int_{\t}|\nabla u|^2-N \int_{\t} F(u)  +\int_{\partial \t}\left((\nabla u \cdot \zeta) \nabla u-\frac{|\nabla u|^2}{2} \zeta+F(u) \zeta\right) \cdot \nu_\t \d\zeta.\label{12eq}
\end{align}
Noting that $\nu_{\rn\smallsetminus\t}=-\nu_\t$ and adding \eqref{12eq} and \eqref{4eq} yields
\begin{align*}
\frac{N-2}{2} \int_{\R^N}|\nabla u|^2-N \int_{\R^N}Q_\t(x)F(u)  + 2\int_{\partial \t}F(u) \zeta\cdot \nu_\t\d \zeta	=0,
\end{align*}
as claimed.
\end{proof}

\begin{corollary}\label{cor:nonexistence}
Let  $\t$ be a possibly unbounded smooth open subset of $\rn$ which is strictly starshaped. If $\rn\smallsetminus\overline{\t}\neq\emptyset$, the problem 
\begin{equation}\label{eq:problem theta}
\begin{cases}
-\Delta u = Q_\t(x)|u|^{2^*-2}u,\\
u\in D^{1,2}(\rn),
\end{cases} 
\end{equation}
admits only the trivial solution.
\end{corollary}

\begin{proof}
If $u$ solves \eqref{eq:problem theta} then, by standard elliptic regularity, $u\in \cC^{1,\alpha}_{loc}(\R^N)$, see \cite[Lemma 3.3]{chs}. Applying Theorem \ref{thm:pohozhaev} with $f(t):=|t|^{2^*-2}t$ we see that $u$ satisfies
\begin{align*}
0=\frac{1}{2^*} \int_{\R^N}|\nabla u|^2-\frac{1}{2^*} \int_{\R^N}Q_\t(x)|u|^{2^*}  + \frac{N-2}{N^2}\int_{\partial \t}|u|^{2^*} \zeta\cdot \nu_\o \d \zeta = \frac{N-2}{N^2}\int_{\partial \t}|u|^{2^*} \zeta\cdot \nu_\t \d \zeta.
\end{align*}
Now, since $\t$ is strictly starshaped, we have that $\zeta\cdot \nu_\t$ is strictly positive on $\partial\t$. This implies that $u=0$ on $\partial \t$. Therefore, $u$ solves
\begin{align*}
-\Delta u = |u|^{2^*-2}u\quad \text{in }\t,\qquad u=0\quad \text{on }\partial \t.
\end{align*}
By the usual Poho\v{z}aev identity and the unique continuation principle this problem admits only the trivial solution, see \cite[Theorem III.1.3]{s}.  Therefore, $u\equiv 0$ in $\t$.  Now, testing \eqref{eq:problem theta} with $u$ we get that
\begin{align*}
0=\int_{\R^N \smallsetminus\t}(|\nabla u|^2+|u|^{2^*}),
\end{align*}
which implies that $u\equiv 0$ in $\R^N$, as claimed.
\end{proof}
	
\begin{proof}[Proof of Theorem \ref{thm:nonexistence}]
$(i)$ was proved in Proposition \ref{prop:minimal energy} and $(ii)$ is a special case of Corollary \ref{cor:nonexistence}.
\end{proof}

\section{Representation of Palais-Smale sequences}
\label{sec:compactness}

In this section we describe the loss of compactness of the functional $J_\lambda$ associated with the problem \eqref{eq:problem}. We shall see that compactness is lost due to the concentration of solutions of the limit problem
\begin{equation} \label{eq:limit problem}
\begin{cases}
-\Delta u = |u|^{2^*-2}u,\\
u\in D^{1,2}(\rn),
\end{cases} 
\end{equation}
\emph{at points that belong to $\o$}. 

We denote by $J_\infty:D^{1,2}(\rn)\to\r$ the functional
$$J_\infty(u):=\frac{1}{2}\|u\|^2 - \frac{1}{2^*}\irn|u|^{2^*},$$
associated with the limit problem \eqref{eq:limit problem}, and by
$$\cN_\infty:=\{u\in D^{1,2}(\rn):u\neq 0, \ \|u\|^2 = \irn|u|^{2^*}\}$$
its Nehari manifold. Then,
$$c_\infty:=\inf_{u\in\cN_\infty}J_\infty(u)=\tfrac{1}{N}S^\frac{N}{2}.$$

For $u\in D^{1,2}(\rn)$, \ $\eps>0$ \ and \ $\xi\in\rn$, \ we define
$$u_{\eps,\xi}(x):=\eps^{\frac{2-N}{2}} u\Big(\frac{x - \xi}{\eps}\Big).$$
Note that $\|u_{\eps,\xi}\|=\|u\|$ and $|u_{\eps,\xi}|_{2^*}=|u|_{2^*},$ where $|\cdot|_{2^*}$ stands for the usual norm in $L^{2^*}(\rn)$.

The proof of the next theorem follows basically Struwe's argument, see \cite[Proposition 2.1]{s1} or \cite[Theorem III.3.1]{s}.

\begin{theorem}\label{thm:struwe}
Let $(u_k)$ be a sequence in $D^{1,2}(\R^N)$ such that
$$J_\lambda(u_k) \to c \qquad \text{and} \qquad J'_\lambda(u_k) \to 0\text{ \ in \ }(D^{1,2}(\rn))'.$$
Passing to a subsequence, there exist a (possibly trivial) solution $u$ of the problem \eqref{eq:problem}, a number $m \in \mathbb{N} \cup \{0\}$, $m$ sequences $(\xi_{1,k}), \dots, (\xi_{m,k})$ of points in $\Omega$, $m$ sequences $(\varepsilon_{1,k}), \dots, (\varepsilon_{m,k})$ in $(0, \infty)$, and $m$ nontrivial solutions $w_1, \dots, w_m$ to the limit problem \eqref{eq:limit problem} with the following properties:
\begin{itemize}
\item[$(i)$] $\lim_{k \to \infty} \varepsilon_{i,k}^{-1} \dist(\xi_{i,k}, \partial \Omega) = \infty$ for each $i = 1, \dots, m$.
\item[$(ii)$] $\lim_{k \to \infty} \|u_k - u - \sum\limits_{i=1}^m [w_i]_{\varepsilon_{i,k}, \xi_{i,k}}\| = 0$.
\item[$(iii)$] $J_\lambda(u) + \sum\limits_{i=1}^m J_\infty(w_i) = c\geq 0$.
\end{itemize}	
\end{theorem}

The following result provides the induction step to prove Theorem \ref{thm:struwe}.

\begin{lemma}\label{lem:struwe}
Let $(u_k)$ be a sequence in $D^{1,2}(\mathbb{R}^{N})$ such that $u_k \rightharpoonup 0$ weakly but not strongly in $D^{1,2}(\mathbb{R}^{N})$, 
$$J_0(u_k) \to c \qquad \text{and} \qquad J'_0(u_k) \to 0\text{ \ in \ }(D^{1,2}(\rn))'.$$
Then, passing to a subsequence, there exist a sequence $(\xi_k)$ in $\Omega$, a sequence $(\varepsilon_k)$ in $(0, \infty)$ and a nontrivial  solution $w$ to the limit problem \eqref{eq:limit problem} with the following properties:
\begin{itemize}
\item[$(i)$] $\varepsilon_k^{-1} \dist(\xi_k, \partial \Omega) \to \infty$ as $k \to \infty$.
\item[$(ii)$] There exists $\delta>0$ such that
$$\int_{B_{\eps_k}(\xi_k)}|u_k|^{2^*}=\delta\qquad\text{for every \ }k\in\n.$$
\item[$(iii)$] $\what{u}_k \rightharpoonup 0$ weakly in $D^{1,2}(\mathbb{R}^{N})$, \ $J_0(\what{u}_k) \to c - J_{\infty}(w)$ \ and \ $J_0'(\what{u}_k) \to 0$, where 
$$\what{u}_k(x):=u_{k}(x)- w_{\eps_k,\xi_k}(x).$$
\end{itemize}
\end{lemma}

\begin{proof}	
Since $u_k \rightharpoonup 0$ weakly in $D^{1,2}(\mathbb{R}^N)$, the sequence $(u_k)$ is bounded in $D^{1,2}(\mathbb{R}^N)$. Therefore, $J'_0(u_k) u_k \to 0$, which implies
\[
o(1) + \| u_k \|^2 = \int_{\mathbb{R}^N} Q_\o(x) |u_k|^{2^*}.
\]
Thus, 
\begin{equation}\label{ecuq1}
c = \lim_{k \to \infty} J_0(u_k) = \lim_{k \to \infty} \frac{1}{N} \int_{\mathbb{R}^{N}} Q_\o(x) |u_k|^{2^*}= \lim_{k \to \infty} \frac{1}{N} \|u_k\|^2 \geq 0
\end{equation}
and, since $u_k$ does not converge strongly to $0$ in $D^{1,2}(\mathbb{R}^N)$, we conclude that $c > 0$. The equation \eqref{ecuq1} allows us to write
\begin{equation}\label{Ecu1}
	Nc+\int_{\mathbb{R}^{N}\smallsetminus \Omega}|u_{k}|^{2^{*}}+o(1)	=\int_{\Omega}|u_{k}|^{2^{*}}.
\end{equation}

\textbf{Step 1:} Selection of the elements \( \xi_k \) and $\eps_k $.

Set $\delta\in (0,\frac{1}{2}S^{N/2})$. For each $k\in\mathbb{N}$, consider the function $\Phi_{k}:[0,\infty)\rightarrow \mathbb{R}$, defined by
\begin{equation*}
\Phi_{k}(r):=\sup_{x\in\mathbb{R}^{N}}\int_{B_{r}(x)}\mathds 1_{\Omega}|u_{k}|^{2^{*}},
\end{equation*}
which is continuous and nondecreasing. Additionally, it satisfies $\Phi_{k}(0)=0$ and $\Phi_{k}(r)=\int_{\Omega}|u_{k}|^{2^{*}}$ for all $r\geq \frac{R}{2}$, where $R$ is the diameter of $\Omega$. From equation \eqref{Ecu1} we conclude that for sufficiently large $k$, there exist  $\eps_{k}\in (0,\frac{R}{2})$ and $\zeta_{k}\in \mathbb{R}^{N}$ such that
\begin{equation*}
0<\delta=\Phi_{k}(\eps_{k})=\sup_{x\in\mathbb{R}^{N}}\int_{B_{\eps_{k}}(x)}\mathds 1_{\Omega}|u_{k}|^{2^{*}}=\int_{B_{\eps_{k}}(\zeta_{k})}\mathds 1_{\Omega}|u_{k}|^{2^{*}}.
\end{equation*}
It follows that $\dist(\zeta_{k},\Omega)\leq \eps_{k}$. Hence, the sequence  $(\zeta_{k})$ is bounded in $\mathbb{R}^{N}$. After passing to a subsequence we get 
\begin{equation*}
d_{k}:=\eps_{k}^{-1}\mbox{dist}(\zeta_{k},\partial\Omega)\rightarrow d\in [0,\infty].
\end{equation*}
We have the following two possibilities:
\begin{itemize}
\item[(1)] If $d\in[0,\infty)$, for each $k$ we choose $\xi_{k}\in \partial\Omega$ such that $\eps_{k}^{-1}|\zeta_{k}-\xi_{k}|=d_{k}$. 
\item [(2)] If $d=\infty$, then $\zeta_{k}\in\Omega$. In this case, we set $\xi_{k}:=\zeta_{k}$.
\end{itemize}
If	$d\in[0,\infty)$ we choose $C_{0}>d+1$,  and if $d=\infty$ we set $C_{0}=1$. In both cases, for sufficiently large $k$, we have that $B_{\eps_{k}}(\zeta_{k})\subset B_{C_{0}\eps_{k}}(\xi_{k})$. Thus,
\begin{equation}\label{Ecu2}
\delta=\sup_{x\in\mathbb{R}^{N}}\int_{B_{\eps_{k}}(x)}1_{\Omega}|u_{k}|^{2^{*}}=\int_{B_{\eps_{k}}(\zeta_{k})}1_{\Omega}|u_{k}|^{2^{*}}\leq \int_{B_{C_{0}\eps_{k}}(\xi_{k})}1_{\Omega}|u_{k}|^{2^{*}}.
\end{equation}

\textbf{Step 2:}	Existence of \( w \neq 0 \). 

We define
\begin{equation}
w_{k}(x) := \eps_{k}^{\frac{N-2}{2}} u_{k}(\eps_{k} x + \xi_{k}).
\end{equation}
By performing a change of variables, from equation \eqref{Ecu2}, we obtain
\begin{equation}\label{Ecu4}
\delta=\sup_{z\in\mathbb{R}^{N}}\int_{B_{1}(z)}\mathds 1_{\Omega_k}|w_{k}|^{2^{*}}\leq \int_{B_{C_{0}}(0)}\mathds 1_{\Omega_k}|w_{k}|^{2^{*}}\leq \int_{B_{C_{0}}(0)}|w_{k}|^{2^{*}},
\end{equation}
where $\Omega_k:=\{z\in\mathbb{R}^{N}: \eps_{k} z+\xi_{k}\in\Omega\}$. Since $\|w_{k}\| = \|u_{k}\|$, \ \( (w_k) \) is a bounded sequence in \( D^{1,2}(\mathbb{R}^N) \) and, passing to a subsequence, we have that \ $w_k \rightharpoonup w$ weakly in $D^{1,2}(\mathbb{R}^N)$, \ $w_k \to w$ in $L_{\text{loc}}^2(\mathbb{R}^N)$ \ and \ $w_k \to w$ a.e. in $\mathbb{R}^N$. 

We will now prove that $w \neq 0$. Arguing by contradiction, suppose $w = 0$. Let $\varphi \in \cC_c^\infty(\mathbb{R}^N)$ and set $\varphi_k(x) := \varphi \big( \frac{x - \xi_k}{\epsilon_k} \big)$. It follows that 
\begin{align}\label{Ecu3}
\int_{\mathbb{R}^{N}} \nabla w_{k} \cdot \nabla (\varphi^{2} w_{k}) - \int_{\mathbb{R}^N}Q_{\Omega_k} \varphi^2 |w_{k}|^{2^{*}} = \int_{\mathbb{R}^N} \nabla u_k \cdot \nabla (\varphi_{k}^2 u_k) - \int_{\mathbb{R}^N}Q_\o|u_{k}|^{2^{*}-2}u_{k}(\varphi_{k}^{2}u_{k})=o(1),
\end{align}
because $(\varphi_{k}^2 u_k)$ is bounded in  $D^{1,2}(\mathbb{R}^{N})$ and $J'_0( u_k) \to 0$. Here we use the notation  \eqref{eq:Q_t}, i.e., $Q_{\o_k}:=\mathds 1_{\o_k} - \mathds 1_{\rn\smallsetminus\o_k}$. If $\varphi \in \cC_{c}^\infty(B_{1}(z))$ for some \( z \in \mathbb{R}^N, \) then, since \( w_k \to 0 \) in \( L^2_{\text{loc}}(\mathbb{R}^N), \) using identity \eqref{Ecu3}, Hölder's inequality, identity \eqref{Ecu4}, and that \( \delta < \frac{1}{2}S^{N/2},\) we obtain 
\begin{align*}
\int_{\mathbb{R}^N} |\nabla (\varphi w_k)|^2 &= \int_{\mathbb{R}^N} \varphi^{2}| \nabla w_k|^2 + \int_{\mathbb{R}^N} 2\varphi w_k \nabla w_{k}\cdot \nabla \varphi + \int_{\mathbb{R}^N} w_k^2 |\nabla \varphi|^2 \\
&=\int_{\mathbb{R}^N} \nabla w_{k}\cdot \nabla (\varphi^{2} w_{k})+o(1)=\int_{B_{1}(z)}Q_{\Omega_k} |\vp w_k|^2|w_{k}|^{2^{*}-2}+o(1)\\
&\leq \int_{B_{1}(z)} \mathds 1_{\o_k} |\varphi w_{k}|^2 |w_{k}|^{2^{*}-2}+o(1)\\
&\leq \Big(\int_{B_{1}(z)} \mathds 1_{\o_k} | w_{k}|^{2^{*}} \Big)^{\frac{2^{*}-2}{2^{*}}}\Big(\int_{B_{1}(z)} \mathds 1_{\o_k} |\varphi w_{k}|^{2^{*}} \Big)^{\frac{2}{2^{*}}} +o(1)\\  
&\leq \delta^\frac{2}{N}S^{-1}\int_{\mathbb{R}^{N}}  |\nabla(\varphi w_{k})|^{2}  +o(1)\\
&< \left(\tfrac{1}{2} \right) ^{\frac{2}{N}}\int_{\mathbb{R}^{N}}  |\nabla(\varphi w_{k})|^{2}  +o(1).
\end{align*}
As a result, $\|\varphi w_k\|^2 = o(1)$ for every $\varphi\in \cC^{\infty}_{c}(B_{1}(z))$ and every $z\in \mathbb{R}^{N}$. Let $z_1, \dots, z_m \in \mathbb{R}^N$ be such that $B_{C_{0}}(0) \subset B_{1/2}(z_1) \cup \dots \cup B_{1/2}(z_m)$ and choose $\varphi_i \in \cC_c^\infty (B_1(z_i))$ such that $\varphi_i \geq 0$ and $\varphi_i(z) = 1$ for all $z \in B_{1/2}(z_i)$. Then,
\[
\int_{B_{C_0}(0)} |w_k|^{2^*} \leq \sum_{i=1}^m \int_{B_1(z_i)} |\varphi_i w_k|^{2^*} = o(1),
\]
which contradicts the statement \eqref{Ecu4}. Thus, $w \neq 0$.

As a consequence, after passing to a subsequence, $\varepsilon_k \to 0$. Indeed, recall that the sequences $(\xi_k)$ and $(\varepsilon_k)$ are bounded. If $\varepsilon_k \to \varepsilon>0$ and $\xi_k \to \xi$, then, as $u_k \rightharpoonup 0$ and $w_k \rightharpoonup w\neq 0$ weakly in $D^{1,2}(\mathbb{R}^N)$, we obtain that
\begin{align*}
0 &< \|w\|^2 = \lim_{k \to \infty} \int_{\mathbb{R}^N} \nabla w_k \cdot \nabla w = \lim_{k \to \infty} \int_{\mathbb{R}^N} \nabla u_k \cdot \nabla w_{\varepsilon_k, \xi_k}\\
		&= \lim_{k \to \infty} \int_{\mathbb{R}^N} \nabla u_k \cdot \nabla w_{\varepsilon, \xi} + \lim_{k \to \infty} \int_{\mathbb{R}^N} \nabla u_k \cdot \nabla (w_{\varepsilon_k, \xi_k} - w_{\varepsilon, \xi}) = o(1),
\end{align*}
which is a contradiction.

\textbf{Step 3:} We show that $\xi_k\in\o$, \ $\eps_k^{-1}\dist(\xi_k,\partial\o)\to\infty$ \ and \ $w$ solves \eqref{eq:limit problem}.

To this end, we analyze the alternatives (1) and (2) stated in step 1. We will show that (1) leads to a contradiction and, thus, (2) must hold true.

Alternative (1) says that $\xi_k \in \partial \Omega$. Then, by passing to a subsequence, we have that $\xi_k \to \xi\in\partial\Omega$. Let $\nu_\xi$ be the outer unit normal vector at $\xi$ and
$$\mathbb{H} := \{ z \in \mathbb{R}^N : \nu_\xi \cdot z < 0 \}.$$
Since $\eps_k\to 0$ we have that $\mathds 1_{\Omega_k}\rightarrow \mathds 1_{\mathbb{H}}$ and $\mathds 1_{\mathbb{R}^{N}\smallsetminus\overline{\Omega}_k}\rightarrow\mathds 1_{\mathbb{R}^{N}\smallsetminus\overline{\mathbb{H}}}$ pointwise in $\rn$.  Hence, for any $\varphi\in \cC_{c}^{\infty}(\mathbb{R}^{N})$, we have that
\begin{align*}
o(1)&=J'_0(u_{k})[\varphi_{\eps_{k},\xi_{k}}]=\int_{\mathbb{R}^{N}}\nabla u_{k}\cdot \nabla\varphi_{\eps_{k},\xi_{k}} -	\int_{\mathbb{R}^{N}}Q_\o|u_{k}|^{2^{*}-2}u_{k}\varphi_{\eps_{k},\xi_{k}}\\
&=\int_{\mathbb{R}^{N}}\nabla w_{k}\cdot \nabla\varphi -	\int_{\mathbb{R}^{N}}Q_{\Omega_k}|w_{k}|^{2^{*}-2}w_{k}\varphi =	\int_{\mathbb{R}^{N}}\nabla w\cdot \nabla\varphi -	\int_{\mathbb{R}^{N}}Q_{\mathbb{H}}|w|^{2^{*}-2}w\varphi.
\end{align*}	
Thus, if (1) holds true, then $w\neq 0$ is a solution to
\begin{equation*}
-\Delta w =Q_{\mathbb{H}}(x) |w|^{2^*-2} w, \qquad w \in D^{1,2}(\R^N).
\end{equation*}
This contradicts Corollary \ref{cor:nonexistence}. 
	
Therefore, alternative (2) must be true, i.e., \( \xi_k \in \Omega \) and \( \eps_k^{-1} \dist(\xi_k, \partial \Omega) \to \infty \). In this case, $\mathds 1_{\Omega_k}\rightarrow\mathds 1_{\mathbb{R}^{N}}$ and $\mathds 1_{\mathbb{R}^{N}\smallsetminus\overline{\Omega}_k}\rightarrow 0$ pointwise in $\rn$. Arguing as above, for any  $\varphi\in \cC_{c}^{\infty}(\mathbb{R}^{N})$, we get that
\begin{equation*}
o(1)=\int_{\mathbb{R}^{N}}\nabla w\cdot \nabla\varphi -	\int_{\mathbb{R}^{N}}|w|^{2^{*}-2}w\varphi.
\end{equation*}	
This proves that $w$ is a solution to the limit problem \eqref{eq:limit problem}, as claimed.

\textbf{Step 4:} We show that $\what{u}_k(x):=u_{k}(x)-w_{\eps_k,\xi_k}(x)$ converges weakly to $0$ in $D^{1,2}(\mathbb{R}^N)$, $J_0(\what{u}_k) \to c - J_{\infty}(w)$ and $ J_0'(\what{u}_k) \to 0$.

Since $u_k \rightharpoonup 0$ and $w_{\eps_{k}, \xi_{k}} \rightharpoonup 0$ weakly in $D^{1,2}(\mathbb{R}^N)$, the first statement is clear. 

On the other hand, as $w_k \rightharpoonup w$ weakly in $D^{1,2}(\rn)$, we get
\begin{equation}\label{eq:J1}
\|u_k\|^2 = \|w_k\|^2 = \|w_k - w\|^2 + \|w\|^2 + o(1)=\|u_k - w_{\varepsilon_k, \xi_k}\|^2 + \|w\|^2 + o(1).
\end{equation}
From the proof of the Brezis-Lieb lemma one sees that the right-hand side of the inequality
$$\left|\irn Q_{\o_k}|w_k|^{2^*}-Q_{\o_k}|w_k-w|^{2^*}-Q_{\o_k}|w|^{2^*}\right|\leq \irn\left||w_k|^{2^*}-|w_k-w|^{2^*}-|w|^{2^*}\right|$$
tends to $0$ as $k\to\infty$, see \cite[Proof of Theorem 2]{bl} or \cite[Proof of Lemma 1.32]{w}. So, as $\mathds 1_{\Omega_k}\rightarrow\mathds 1_{\mathbb{R}^{N}}$, we obtain
\begin{align}\label{eq:J2}
\irn Q_\o|u_{k}|^{2^*}&=\irn Q_{\Omega_{k}}|w_{k}|^{2^*}=\irn Q_{\Omega_k}|w_{k}-w|^{2^*}+\irn Q_{\Omega_k}|w|^{2^*}+o(1)\nonumber\\
&=\irn Q_\o|u_{k}-w_{\eps_{k},\xi_{k}}|^{2^*}+\irn |w|^{2^*}+o(1).
\end{align}
Equations \eqref{eq:J1} and \eqref{eq:J2} yield
\begin{equation*}
	c+o(1)=J_0(u_k)=J_0(u_k - w_{\eps_k, \xi_k})+ J_\infty(w)+o(1).
\end{equation*}
Finally, following the proof of \cite[Lemma 8.9]{w} one shows that
$$Q_{\Omega_k}|w_k|^{2^*-2} w_k- Q_{\Omega_k}|w_k - w|^{2^*-2} (w_k - w) - Q_{\Omega_k}|w|^{2^*-2} w\to 0\quad\text{in \ }(D^{1,2}(\rn))'.$$
Then, recalling that $w$ solves \eqref{eq:limit problem}, we obtain that
$$0=\lim_{k\to\infty}J_0'(u_k)=\lim_{k\to\infty}J_0'(u_k-w_{\eps_k,\xi_k})+J_\infty(w)=\lim_{k\to\infty}J_0'(u_k-w_{\eps_k,\xi_k})\quad\text{in \ }(D^{1,2}(\rn))'.$$
This completes the proof.
\end{proof}
\smallskip

\begin{proof}[Proof of Theorem \ref{thm:struwe}]
Let $(u_k)$ be a sequence in $D^{1,2}(\R^N)$ such that $J_\lambda(u_k) \to c$ and $J'_\lambda(u_k) \to 0$. Then, there exists a constant $C > 0$ such that
\[
\tfrac{2}{N-2} \|u_k\|^2_\lambda = 2^* J_\lambda(u_k) - J'_\lambda(u_k) u_k \leq C(1 + \|u_k\|_\lambda) \quad \forall k \in \mathbb{N}.
\]
This shows that the sequence $(u_k)$ is bounded in $D^{1,2}(\R^N)$. Therefore, $J'_\lambda(u_k) u_k \to 0$ and, as a consequence,
\begin{equation}\label{eq:c>0}
c = \lim_{k \to \infty} J_{\lambda}(u_k)= \lim_{k \to \infty} \frac{1}{N} \int_{\mathbb{R}^{N}} Q_\o |u_k|^{2^*}=\lim_{k \to \infty} \frac{1}{N} \|u_k\|^2_\lambda \geq 0.
\end{equation}
Besides, a subsequence of $(u_k)$ satisfies $u_k \rightharpoonup u $ weakly in $D^{1,2}(\R^N)$, $u_k \to u$ in $L_{\text{loc}}^2(\mathbb{R}^N)$ and $u_k \to u$ a.e. in $\mathbb{R}^N$. Therefore, 
\begin{equation*}
0 = \lim_{k \to \infty} J'_\lambda(u_k)\varphi=J'_\lambda(u)\varphi\qquad\text{for every \ }\vp\in\cC^\infty_c(\rn).
\end{equation*} 
This shows that \( u \) is a solution to problem \eqref{eq:problem}, which could be the trivial one.

Define \( v_{1,k} := u_k - u \). Then, $v_{1,k}\rightharpoonup 0$ weakly in $D^{1,2}(\R^N)$, $v_{1,k}\to 0$ in $L_{\text{loc}}^2(\mathbb{R}^N)$ and $v_{1,k}\to 0$ a.e. in $\mathbb{R}^N$. Since $\o$ is bounded,
\begin{equation*}
\lim_{k \to \infty}  \int_{\mathbb{R}^N}\lambda\mathds 1_{\Omega} u_k^2= \int_{\mathbb{R}^N} \lambda\mathds 1_{\Omega} u^2.
\end{equation*}
Then, employing similar arguments to those in step 4 of the proof of Lemma \ref{lem:struwe} we show that
\begin{equation*}
c+o(1)=J_\lambda(u_k)=J_0(v_{1,k})+J_\lambda(u)+o(1)\qquad
\end{equation*}
and
$$0=\lim_{k \to \infty}J_\lambda'(u_k)=\lim_{k \to \infty}J_0'(v_{1,k})+J_\lambda'(u)=\lim_{k \to \infty}J_0'(v_{1,k}),$$
i.e., $J_0(v_{1,k})\to c-J_\lambda(u)=:c_1$ and $J_0'(v_{1,k})\to 0$ in $(D^{1,2}(\R^N))'$.

If \( v_{1,k} \to 0 \) strongly in $D^{1,2}(\R^N)$, then  $u_k \to u$ strongly in $D^{1,2}(\R^N)$ and  $J_\lambda(u) = c$. Thus, Theorem \ref{thm:struwe} holds true with $m = 0$. 

If $v_{1,k} \rightharpoonup 0$ weakly but not strongly in $D^{1,2}(\rn)$, then $c_1>0$. We apply Lemma \ref{lem:struwe} to \( (v_{1,k}) \) and continue applying it until we obtain a sequence $(v_{m,k})$ that converges strongly to $0$. This process must terminate after a finite number of steps, since at each step, the energy level $c$ decreases by at least $\frac{1}{N}S^\frac{N}{2}$ and $c$ cannot be negative, as shown in \eqref{eq:c>0}.
\end{proof}
\smallskip

\begin{proof}[Proof of Theorem \ref{thm:bn}]
Let $\lambda\in(-\Lambda_\o,0)$ and $(u_k)$ be a sequence in $\mathcal{N}_\lambda$ such that  $J_\lambda(u_k)\to c_\lambda$. By Ekeland's variational principle, we may assume that $J'_\lambda(u_k)\to 0$ in $(D^{1,2}(\R^N))'$, see \cite[Theorem 8.5]{w}.

Since $\lambda\in(-\Lambda_\o,0)$, Proposition \ref{prop:minimal energy} says that $c_\lambda<\frac{1}{N}S^\frac{N}{2}$. As $J_\infty(w)\geq\frac{1}{N}S^\frac{N}{2}$ for every nontrivial solution $w$ to the limit problem \eqref{eq:limit problem}, Theorem \ref{thm:struwe}$(iii)$ yields that $m=0$. Then, by Theorem \ref{thm:struwe}$(ii)$, we have that $u_k\to u$ strongly in $D^{1,2}(\rn)$. Therefore, $J_\lambda(u)=c_\lambda$ and $J_\lambda'(u)=0$.

As $|u|$ also satisfies $J_\lambda(|u|)=c_\lambda$ and $J_\lambda'(|u|)=0$, we may assume that $u\geq 0$. By standard elliptic regularity (see \cite[Lemma 3.3]{chs}), $u\in W^{2,r}_{loc}(\rn)\cap \cC^{1,\alpha}_{loc}(\rn)$ for all $r\in[1,\infty)$ and $\alpha\in(0,1)$. Then, the strong maximum principle for strong solutions \cite[Theorem 9.6]{gt} applied to the equation
$$-\Delta u + \mathds 1_{\rn\smallsetminus\o} u^{2^*-1} = -\lambda\mathds 1_\o u + \mathds 1_\o u^{2^*-1}$$
yields that $u>0$ in $\rn$.
\end{proof}

For $\lambda=0$ Theorem \ref{thm:struwe} yields the following results.

\begin{corollary}\label{cor:ps}
Let $\lambda=0$. Then, every sequence $(u_k)$ in $D^{1,2}(\rn)$ such that 
$$J_0(u_k)\to c\in\big(\tfrac{1}{N}S^\frac{N}{2},\tfrac{2}{N}S^\frac{N}{2}\big)\qquad\text{and}\qquad J'_0(u_k)\to 0\text{ \ in \ }(D^{1,2}(\rn))'$$
contains a convergent subsequence.
\end{corollary}

\begin{proof}
Since the limit problem \eqref{eq:limit problem} does not have a solution $w$ with $J_\infty(w)\in \big(\frac{1}{N}S^\frac{N}{2},\frac{2}{N}S^\frac{N}{2}\big)$, the only alternative left by Theorem \ref{thm:struwe} is that the problem \eqref{eq:problem} with $\lambda=0$ has a solution $u$ such that, after passing to a subsequence, $\|u_k-u\|\to 0$, as claimed.
\end{proof}

It is well known that the least energy solutions to the limit problem \eqref{eq:limit problem} are the positive bubbles
\begin{equation}\label{eq:bubble}
U_{\eps,\xi}(x):=(N(N-2))^\frac{N-2}{4}\frac{\eps^{{\frac{N-2}{2}}}}{(\eps^2+|x-\xi|^2)^\frac{N-2}{2}},
\qquad \xi\in\rn, \ \eps>0,
\end{equation}
and the negative ones $-U_{\eps,\xi}$.

\begin{corollary} \label{cor:minimizing}
Let $\lambda=0$ and $(u_k)$ be a sequence in $\mathcal{N}_0$ such that $J_0(u_k)\to c_0=\frac{1}{N}S^\frac{N}{2}$. Then, after passing to a subsequence, there exist $\eps_k>0$ and $\xi_k\in\o$ such that
$$\lim_{k \to \infty} \eps_k^{-1} \dist(\xi_k, \partial \Omega) = \infty$$
and
$$\text{either}\quad\lim_{k \to \infty} \|u_k - U_{\eps_k, \xi_k}\| = 0,\quad\text{or}\quad\lim_{k \to \infty} \|u_k + U_{\eps_k, \xi_k}\| = 0.$$
\end{corollary}

\begin{proof}
By Ekeland's variational principle \cite[Theorem 8.5]{w}, we may assume that $J'_0(u_k)\to 0$ in $(D^{1,2}(\R^N))'$. As shown in Proposition \ref{prop:minimal energy}, $c_0$ is not attained. The statement follows from Theorem \ref{thm:struwe}.
\end{proof}

\section{Multiple solutions for $\lambda=0$}
\label{sec:coron}

In this section we prove Theorem \ref{thm:coron}. 

We recall the notion of Lusternik-Schnirelmann category of a map of pairs. A pair of spaces $(X,A)$ consists of a topological space $X$ and a subspace $A$ of $X$. A map of pairs $f:(X,A)\to(Y,B)$ is a continuous function $f:X\to Y$ such that $f(A)\subset B$. Two maps of pairs $f_0,f_1:(X,A)\to(Y,B)$ are homotopic if there exists a map of pairs $F:(X\times[0,1],A\times[0,1])\to(Y,B)$ such that $F(x,0)=f_0(x)$ and  $F(x,1)=f_1(x)$ for every $x\in X$.

\begin{definition}
The \emph{category of a map of pairs $f:(X,A)\to(Y,B)$}, denoted $\cat(f)$, is the smallest number $m$ such that there exist $m+1$ open subsets $U_0,U_1,\ldots,U_m$ of $X$ with the following properties:
\begin{itemize}
\item $A\subset U_0$ and there is a homotopy of pairs $F:(U_0\times[0,1],A\times[0,1])\to(Y,B)$ such that $F(x,0)=f(x)$ and $F(x,1)\in B$ for every $x\in U_0$,
\item $f|_{U_i}:U_i\to Y$ is homotopic to a constant function for each $i=1,\ldots,m$,
\item $X=U_0\cup U_1\cup\cdots\cup U_m$.
\end{itemize}
If there are no subsets of $X$ with these properties, we set $\cat(f):=\infty$. The \emph{category of $(X,A)$}, denoted $\cat(X,A)$ is the category of the identity map $\mathrm{id}_{(X,A)}$.
\end{definition}

The following property and others can be found in \cite[Section 2]{cp}.

\begin{lemma}\label{lem:category}
If $f:(X,A)\to(Y,B)$ and $g:(Y,B)\to(Z,C)$ are maps of pairs, then
$$\cat(g\circ f)\leq\cat(Y,B).$$
\end{lemma}

\begin{proof}
The proof is straightforward.
\end{proof}

Given $d\in\r$, set
$$\cN_0^{\leq d}:=\{u\in\cN_0:J_0(u)\leq d\}.$$
Lusternik-Schnirelmann theory yields the following result.

\begin{lemma} \label{lem:ls theory}
If $\frac{1}{N}S^\frac{N}{2}<a<b<\frac{2}{N}S^\frac{N}{2}$, then the functional $J_0$ has at least 
$$\cat\big(\cN_0^{\leq b},\cN_0^{\leq a}\big)$$
pairs $\pm u$ of critical points with $J_0(u)\in[a,b]$.
\end{lemma}

\begin{proof}
By Corollary \ref{cor:ps}, $J_0$ satisfies the Palais-Smale condition in the interval $(\frac{1}{N}S^\frac{N}{2},\frac{2}{N}S^\frac{N}{2})$. The statement follows from standard Lusternik-Schnirelmann theory; see, for instance, \cite[Theorem 5.19]{w} or \cite[Theorem 1.1$(a)$]{cp}.
\end{proof}

For any two numbers $0<t<s$ and $X\subset\rn$, define
\begin{align*}
\bar{B}_s X:=\{x\in\rn:\dist(x,M)\leq s\},\qquad S_s X=\partial (\bar B_s X),\qquad A_{t,s}X:=\{x\in\rn:t\leq \dist(x,X)\leq s\}.
\end{align*}

Fix a closed smooth submanifold $M$ of $\rn$ and $R>0$ such that $\bar B_RM$ is a tubular neighborhood of $M$. Recall that $\bar B_RM$ is a (closed) tubular neighborhood of $M$ if for any two different points $x,y\in M$ the closed normal balls of radius $R$ at $x$ and $y$ have empty intersection; see, for instance, \cite[Section 11]{ms}.
 
The first step in proving Theorem \ref{thm:coron} is to define the number $r\in(0,R)$. This is the purpose of the following statement.

\begin{lemma}\label{lem:r}
Given $b\in(\frac{1}{N}S^\frac{N}{2},\frac{2}{N}S^\frac{N}{2})$, there exists $r\in(0,R)$ such that, for any $\vr>0$ and $a\in(\frac{1}{N}S^\frac{N}{2},b)$, there is a continuous function
$$\bar B_RM\to D^{1,2}(\rn),\qquad\xi\mapsto w_\xi,$$
with the following properties:
\begin{itemize}
\item[$(i)$] $w_\xi\in\cC^\infty_c(\rn)$, $w_\xi\neq 0$, $w_\xi$ is radial with respect to $\xi$ and $\supp(w_\xi)\subset\bar B_{R+\vr}M$ for every $\xi\in \bar B_RM$.
\item[$(ii)$] For each $\xi\in \bar B_RM$, 
\begin{equation}
\irn Q_{A_{r,R+\vr}M}|w_\xi|^{2^*}>0\qquad\text{and}\qquad
\frac{\|w_\xi\|^2}{\Big(\irn Q_{A_{r,R+\vr}M}|w_\xi|^{2^*}\Big)^{2/2^*}}\leq
\begin{cases}
(Nb)^\frac{2}{N} &\text{if \ }\xi\in \bar B_RM,\\
(Na)^\frac{2}{N} &\text{if \ }\xi\in S_RM.
\end{cases}
\end{equation}
\end{itemize}
\end{lemma}

\begin{proof}
Fix a radial cut-off function $\chi\in\cC^\infty_c(\rn)$ such that $\chi(x)=1$ if $|x|\leq\frac{R}{8}$ and $\chi(x)=0$ if $|x|\geq\frac{R}{4}$. For each $\eps>0$ set $u_\eps(x):=\chi(x)U_{\eps,0}(x)$, where $U_{\eps,0}$ is the bubble defined in \eqref{eq:bubble}. Using the Brezis-Nirenberg estimates \cite[Lemma 1.1]{bn} it is easy to see that there exists  $C_0>0$ such that, for every $\eps>0$ sufficiently small,
\begin{equation}\label{eq:C_0}
\frac{\|u_\eps\|^2}{|u_\eps|_{2^*}^2}\leq S + C_0\eps^{N-2}.
\end{equation}
Fix $\eps_0>0$ such that $S + C_0\eps_0^{N-2}<(Nb)^\frac{2}{N}$ and, for each $\xi\in \rn$, define
$$v_\xi(x):=u_{\eps_0}(x-\xi).$$
Then,
\begin{equation*}
\supp(v_\xi)\subset\bar B_{R}M\quad\text{for all \ }\xi\in\bar B_\frac{3R}{4}M\qquad\text{and}\qquad\frac{\|v_\xi\|^2}{|v_\xi|_{2^*}^2}<(Nb)^\frac{2}{N}\quad\text{for all \ }\xi\in \rn.
\end{equation*}
Since for any $r\in(0,R)$ and $\xi\in\bar B_\frac{3R}{4}M$,
$$\irn Q_{A_{r,R}}|v_\xi|^{2^*}=|v_\xi|^{2^*}_{2^*} - 2\int_{\bar B_rM}|v_\xi|^{2^*}\qquad\text{and}\qquad\int_{\bar B_rM}|v_\xi|^{2^*}\to 0\text{ \ as \ }r\to 0,$$
we may choose $r\in(0,\frac{R}{4})$ such that
\begin{equation}\label{eq:coron1}
\irn Q_{A_{r,R}}|v_\xi|^{2^*}>0\qquad\text{and}\qquad\frac{\|v_\xi\|^2}{\Big(\irn Q_{A_{r,R}}|v_\xi|^{2^*}\Big)^{2/2^*}}\leq(Nb)^\frac{2}{N}\qquad\text{for all \ }\xi\in B_\frac{3R}{4}M.
\end{equation}
Let us check that this $r$ satisfies the statement of the lemma.

Let $\vr\in(0,\frac{R}{4}]$ and $a\in(\frac{1}{N}S^\frac{N}{2},b)$ be given. Fix $\eps_1\in(0,\eps_0]$ such that $S + C_0\eps_1^{N-2}\leq (Na)^\frac{2}{N}$ and set
\begin{equation*}
t(\xi):=
\begin{cases}
0 &\text{if \ }\xi\in\bar B_\frac{R}{2}M,\\
\frac{4}{R}(\dist(\xi,M)-\frac{R}{2}) &\text{if \ }\xi\in A_{\frac{R}{2},\frac{3R}{4}}M,\\
1  &\text{if \ }\xi\in A_{\frac{3R}{4},R}M,
\end{cases}
\end{equation*}
$$\eps(\xi):=(1-t(\xi))\eps_0+t(\xi)\eps_1,\qquad u_{\eps(\xi)}:=\chi U_{\eps(\xi),0},\qquad\vr(\xi):=\tfrac{4}{R}\big(\tfrac{4}{R}\vr-1\big)\big(\dist(\xi,M)-\tfrac{3R}{4}\big)+1.$$
Define
\begin{equation*}
w_\xi(x):=
\begin{cases}
u_{\eps(\xi)}(x-\xi) & \text{if \ }\xi\in\bar B_\frac{3R}{4}M, \smallskip\\
\vr(\xi)^\frac{2-N}{2}u_{\eps_1}\big(\frac{x-\xi}{\vr(\xi)})& \text{if \ }\xi\in A_{\frac{3R}{4},R}M.
\end{cases}
\end{equation*}
Then, $\supp(w_\xi)\subset\bar B_\frac{R}{4}(\xi)$ if $\xi\in \bar B_\frac{3R}{4}M$ and $\supp(w_\xi)\subset\bar B_{\frac{R}{4}\vr(\xi)}(\xi)$ if $\xi\in A_{\frac{3R}{4},R}M$. Therefore, $\supp(w_\xi)\subset\bar B_{R+\vr}M$ and $w_\xi$ satisfies $(i)$.

If $\xi\in A_{\frac{R}{2},R}M$ then, as $r\in(0,\frac{R}{4})$, we have that $\supp(w_\xi)\subset A_{r,R+\vr}M$. As a consequence,
$$\frac{\|w_\xi\|^2}{\big(\irn Q_{A_{r,R+\vr}M}|w_\xi|^{2^*}\big)^{2/2^*}}=\frac{\|w_\xi\|^2}{|w_\xi|_{2^*}^2}=\frac{\|u_{\eps(\xi)}\|^2}{|u_{\eps(\xi)}|_{2^*}^2}\leq S + C_0\eps(\xi)^{N-2}\leq S + C_0\eps_0^{N-2}\leq (Nb)^\frac{2}{N}\quad\text{if \ }\xi\in A_{\frac{R}{2},\frac{3R}{4}}M$$
and, using the invariance of both norms under dilations,
$$\frac{\|w_\xi\|^2}{\big(\irn Q_{A_{r,R+\vr}M}|w_\xi|^{2^*}\big)^{2/2^*}}=\frac{\|w_\xi\|^2}{|w_\xi|_{2^*}^2}=\frac{\|u_{\eps_1}\|^2}{|u_{\eps_1}|_{2^*}^2}\leq S + C_0\eps_1^{N-2}\leq (Na)^\frac{2}{N}\quad\text{if \ }\xi\in A_{\frac{3R}{4},R}M.$$
Since $w_\xi=v_\xi$ if $\xi\in\bar B_\frac{R}{2}M$, these two inequalities, together with \eqref{eq:coron1}, yield $(ii)$. This completes the proof.
\end{proof}

Let
$$\beta_\o:L^{2^*}(\o)\smallsetminus\{0\}\to\rn,\qquad\beta_\o(v):=\frac{\io x\,|v(x)|^{2^*}\d x}{\io|v(x)|^{2^*}\d x}$$
be the barycenter map in $\o$. Since $\o$ is bounded, this function is well defined and continuous. Define
$$\beta:\cN_0\to\rn,\qquad\beta(u):=\beta_\o(\mathds 1_\o u).$$
Noting that
\begin{equation*}
S^\frac{N}{2}\leq\|u\|^2\leq\io|u|^{2^*}\qquad\text{for every \ }u\in\cN_0,
\end{equation*}
we see that $\beta$ is well defined and that it is continuous.

\begin{lemma}\label{lem:a}
Given $\vr>0$, there exists $a\in(\frac{1}{N}S^\frac{N}{2},\frac{2}{N}S^\frac{N}{2})$ such that
$$\dist(\beta(u),\o)<\vr\quad\text{for every \ }u\in\cN_0^{\leq a}.$$
\end{lemma}

\begin{proof}
Arguing by contradiction, assume there exist $\vr_0>0$ and $u_k\in\cN_0$ such that $J_0(u_k)\leq \frac{1}{N}S^\frac{N}{2}+\frac{1}{k}$ and
\begin{equation}\label{eq:contradiction}
\dist(\beta(u_k),\o)\geq\vr_0\quad\text{for every \ }k\in\mathbb{N}.
\end{equation}
Then, by Corollary \ref{cor:minimizing}, replacing $u_k$ with $-u_k$ if necessary and passing to a subsequence, there exist $\eps_k>0$ and $\xi_k\in\o$ such that
$$\lim_{k\to\infty}\eps_k^{-1}\dist(\xi_k,\partial\o) = \infty\qquad\text{and}\qquad\lim_{k\to\infty}\|u_k - U_{\eps_k,\xi_k}\| = 0.$$
Let $d_k:=\dist(\xi_k,\partial\o)$ and set $v_k:=\mathds 1_{B_{d_k}(\xi_k)}U_{\eps_k,\xi_k}$. Since $v_k$ is radial with respect to $\xi_k$ and $\supp(v_k)\subset\overline{\o}$, its barycenter is $\beta_\o(v_k)=\xi_k$. A change of variable yields
\begin{align} 
\io |v_k|^{2^*}&=\int_{B_{d_k}(\xi_k)}U_{\eps_k,\xi_k}^{2^*}=\int_{B_{\eps_k^{-1}d_k}(0)}U_{1,0}^{2^*}= S^\frac{N}{2}+o(1),\label{eq:lbd1}\\
\int_{\rn\smallsetminus\o}|u_k|^{2^*}&\leq C\Big(|u_k-U_{\eps_k,\xi_k}|_{2^*}^{2^*} + \int_{\rn\smallsetminus\o}U_{\eps_k,\xi_k}^{2^*}\Big) \leq C\|u_k-U_{\eps_k,\xi_k}\|^{2^*} + C\int_{\rn\smallsetminus B_{\eps_k^{-1}d_k(0)}}U_{1,0}^{2^*}=o(1),\nonumber\\
|u_k-v_k|_{2^*}^{2^*}&\leq C\left(|u_k-U_{\eps_k,\xi_k}|_{2^*}^{2^*}+|U_{\eps_k,\xi_k}-v_k|_{2^*}^{2^*}\right)\leq C\|u_k-U_{\eps_k,\xi_k}\|^{2^*} + C\int_{\rn\smallsetminus B_{\eps_k^{-1}d_k(0)}}U_{1,0}^{2^*}=o(1).\nonumber
\end{align}
As a consequence,
\begin{align}
\io|u_k|^{2^*}&=\irn Q_\o|u_k|^{2^*}+\int_{\rn\smallsetminus\o}|u_k|^{2^*}=S^\frac{N}{2} + o(1),\label{eq:lbd2} \\
\irn\big||u_k|^{2^*}-|v_k|^{2^*}\big|&\leq 2^*\irn\big(|u_k|+|v_k|\big)^{2^*-1}|u_k-v_k|\leq 2^*\big(|u_k|_{2^*}+|v_k|_{2^*}\big)^{2^*-1}|u_k-v_k|_{2^*}=o(1). \label{eq:lbd3}
\end{align}
Since $\beta_\o(v_k)=\xi_k\in\o$ and $\o$ is bounded, using \eqref{eq:lbd1}, \eqref{eq:lbd2} and \eqref{eq:lbd3} we obtain
\begin{align*}
|\beta(u_k)-\xi_k|&=\left|\frac{\io x\,(|u_k(x)|^{2^*}-|v_k(x)|^{2^*})\d x}{\io|u_k(x)|^{2^*}\d x} + \left(\frac{\io|v_k(x)|^{2^*}\d x}{\io|u_k(x)|^{2^*}\d x}\right)\beta_\o(v_k) - \beta_\o(v_k)\right|\\
&\leq C\left(\irn\big||u_k|^{2^*}-|v_k|^{2^*}\big| + \left(\frac{\io|v_k(x)|^{2^*}\d x}{\io|u_k(x)|^{2^*}\d x}-1\right)\right)=o(1).
\end{align*}
This contradicts \eqref{eq:contradiction}.
\end{proof}

The following theorem gives an estimate for the number of positive solutions to problem \eqref{eq:problem} with $\lambda=0$ in terms of the Lusternik-Schnirelmann category.

\begin{theorem}\label{thm:coron_ls}
Given $b\in(\frac{1}{N}S^\frac{N}{2},\frac{2}{N}S^\frac{N}{2})$, a closed submanifold $M$ of $\rn$ and a tubular neighborhood $\bar B_RM$ of $M$, there exists $r\in(0,R)$ such that, if $\o$ satisfies
\begin{equation}\label{eq:assumptions_ls}
M\cap\overline{\o}=\emptyset\qquad\text{and}\qquad A_{r,R}M\subset\o,
\end{equation}
then the problem \eqref{eq:problem} with $\lambda=0$ has at least 
$$\cat\big[(\bar B_RM,S_RM)\hookrightarrow (\rn,\rn\smallsetminus M)\big]$$
positive solutions $u$ with $J_0(u)\in(\frac{1}{N}S^\frac{N}{2},b\,]$.
\end{theorem}

\begin{proof}
Let $r\in(0,R)$ be given by Lemma \ref{lem:r}. Since $\o$ satisfies \eqref{eq:assumptions_ls}, 
$$\vr:=\min\{\dist(M,\o),\,\dist(A_{r,R}M,\rn\smallsetminus\o\}>0.$$
For this $\vr$ we choose $a\in(\frac{1}{N}S^\frac{N}{2},b)$ as in Lemma \ref{lem:a}. Then, $\beta$ is a map of pairs
$$\beta:(\cN_0^{\leq b},\cN_0^{\leq a})\to(\rn,\rn\smallsetminus M).$$
Furthermore, Lemma \ref{lem:r} yields a continuous function 
$$\bar B_RM\to D^{1,2}(\rn)\smallsetminus\{0\},\qquad\xi\mapsto w_\xi,$$
that satisfies the statements $(i)$ and $(ii)$ of that lemma. As observed in \eqref{eq:U}, for every $\xi\in \bar B_RM$ there exists $\tau_\xi\in(0,\infty)$ such that $\tau_\xi w_\xi\in\cN_0$ and, since $\o\supset A_{r,R+\vr}M$,
\begin{equation*}
J_0(\tau_\xi w_\xi)=\frac{1}{N}\left(\frac{\|w_\xi\|^2}{\Big(\irn Q_\o|w_\xi|^{2^*}\Big)^{2/2^*}}\right)^\frac{N}{2}\leq\frac{1}{N}\left(\frac{\|w_\xi\|^2}{\Big(\irn Q_{A_{r,R+\vr}M}|w_\xi|^{2^*}\Big)^{2/2^*}}\right)^\frac{N}{2}\leq
\begin{cases}
b &\text{if \ }\xi\in \bar B_RM,\\
a &\text{if \ }\xi\in S_RM.
\end{cases}
\end{equation*}
This gives a map of pairs
$$\iota:(\bar B_RM,S_RM)\to (\cN_0^{\leq b},\cN_0^{\leq a}),\qquad \iota(\xi):=\tau_\xi w_\xi.$$
Since $w_\xi$ is radial with respect to $\xi$, we have that $\beta(\tau_\xi w_\xi)=\xi$ for every $\xi\in\bar B_RM$. Then, Lemma \ref{lem:ls theory} yields
$$\cat\big[(\bar B_RM,S_RM)\hookrightarrow (\rn,\rn\smallsetminus M)\big]\leq\cat\big(\cN_0^{\leq b},\cN_0^{\leq a}\big)$$
and, from Lemma \ref{lem:category}, we get at least 
$$\cat[(\bar B_RM,S_RM)\hookrightarrow (\rn,\rn\smallsetminus M)]$$
pairs of solutions $\pm u$ to problem \eqref{eq:problem} in $\cN_0^{\leq b}$. As $b<\frac{2}{N}S^\frac{N}{2}=2c_0$, a well known argument shows that these solutions do not change sign; see, for instance, \cite[Proof of Theorem A]{bc}. So, assuming that $u\geq 0$ and arguing as in the proof of Theorem \ref{thm:bn}, we conclude that $u>0$ in $\rn$.
\end{proof}

The cup-length is a lower bound for the category. We state this fact precisely below. For convenience, we use singular cohomology $H^*(\,\cdot\,;\z2)$ with $\z2$-coefficients; see Remark \ref{rem:z2}.

\begin{definition} \label{def:cup length}
The \emph{cup-length of a map of pairs $f:(X,A)\to(Y,B)$}, denoted $\cupl(f)$, is the smallest number $m$ such that
$$f^*(\gamma_1\smile\cdots\smile\gamma_m\smile\alpha)=0\qquad\text{for any \ }\alpha\in H^*(Y,B;\z2)\text{ \ and \ }\gamma_1,\ldots,\gamma_m\in\tilde H^*(Y;\z2),$$
where $\tilde H^*(\,\cdot\,;\z2)$ is reduced cohomology. The \emph{cup-length of $(X,A)$}, denoted $\cupl(X,A)$ is the cup-length of the identity map $\mathrm{id}_{(X,A)}$.
\end{definition}

Note that, if $A=\emptyset$, then, as $H^0(X;\z2)$ has a unit, the cup-length of $X$ is the smallest number $m$ such that $\gamma_1\smile\cdots\smile\gamma_m=0$ for any $m$ cohomology classes in $\tilde H^*(X;\z2)$.

\begin{lemma}\label{lem:cup length}
For any map of pairs $f:(X,A)\to(Y,B)$,
$$\cupl(f)\leq \cat(f).$$
\end{lemma}

\begin{proof}
The proof is an easy consequence of the exactness and homotopy properties of cohomology and the definition of the cup product; see \cite[Proposition 4.3]{cp}.
\end{proof}
\smallskip

\begin{proof}[Proof of Theorem \ref{thm:coron}]
By Theorem \ref{thm:coron_ls}, it suffices to show that
\begin{equation}\label{eq:cl}
\cupl(M)\leq \cat\big[(\bar B_RM,S_RM)\hookrightarrow (\rn,\rn\smallsetminus M)\big].
\end{equation}
From Lemma \ref{lem:cup length} we get that
\begin{equation}\label{eq:cl1}
\cupl[(\bar B_RM,S_RM)\hookrightarrow (\rn,\rn\smallsetminus M)]\leq\cat[(\bar B_RM,S_RM)\hookrightarrow (\rn,\rn\smallsetminus M)].
\end{equation}
Using the exactness and the homotopy properties of cohomology and the five-lemma, we see that the first inclusion in the next line
$$(\bar B_RM,S_RM)\hookrightarrow(\bar B_RM,\bar B_RM\smallsetminus M)\hookrightarrow (\rn,\rn\smallsetminus M)$$
induces an isomorphism in cohomology, whereas the second one induces an isomorphism by excision. As a consequence,
\begin{equation}\label{eq:cl2}
\cupl[(\bar B_RM,S_RM)\hookrightarrow (\rn,\rn\smallsetminus M)]=\cupl(\bar B_RM,S_RM).
\end{equation}
The Thom isomorphism theorem says that there exists a cohomology class $\vt\in H^d(\bar B_RM,S_RM;\z2)$ in dimension $d:=N-\dim M$, called \emph{Thom class}, such that the homomorphism
$$H^j(M;\z2)\equiv H^j(\bar B_RM;\z2)\to H^{d+j}(\bar B_RM,S_RM;\z2),\qquad\gamma\mapsto \gamma\smile\vt,$$
is an isomorphism; see, for instance, \cite[Theorem 8.1]{ms}. As a consequence, if $\gamma_1,\ldots,\gamma_n\in\tilde H^*(M;\z2)$ satisfy $\gamma_1\smile\cdots\smile\gamma_n\neq 0$, then $\gamma_1\smile\cdots\smile\gamma_n\smile\vt\neq 0$. This shows that
$$\cupl(M)\leq \cupl(\bar B_RM,S_RM).$$
Combining this inequality with \eqref{eq:cl1} and \eqref{eq:cl2} we obtain \eqref{eq:cl}. This completes the proof.
\end{proof}

\begin{remark}\label{rem:z2}
\emph{We have stated Theorem \ref{thm:coron} in terms of singular cohomology with $\z2$-coefficients to avoid technicalities arising from the orientability in the Thom isomorphism theorem. Under suitable assumptions, the result holds for other multiplicative cohomology theories as well. }
\end{remark}

\section{The effect of symmetries}
\label{sec:symmetries}

In this section we prove Theorem \ref{thm:sym1}. 
        
Let \( G \) be a closed subgroup of the group \( O(N) \) of linear isometries of $\rn$. The \( G \)-orbit of a point \( x \in \mathbb{R}^N \) is
        \[
        Gx := \{ gx : g \in G \}.
        \]
We denote its cardinality by \( \#Gx \). A subset \( \Omega \subset \mathbb{R}^N \) is called \( G \)-invariant if \( Gx \subset \Omega \) for all \( x \in \Omega \), and a function \( u : \mathbb{R}^N \to \mathbb{R} \) is \( G \)-invariant if it is constant on each \( G \)-orbit \( Gx \).
        
If \( g \in G \) and $u \in D^{1,2}(\rn)$, we define $gu \in D^{1,2}(\rn)$ by $gu(x) := u(g^{-1} x)$. This defines an isometric action of \( G \) on $D^{1,2}(\rn)$. The $ G $-fixed point space of  $D^{1,2}(\rn)$ is the subspace
        \[
        D^{1,2}(\rn)^{G} := \{ u \in D^{1,2}(\rn) : u \text{ is } G\text{-invariant} \}.
        \]
If we assume that \( \Omega \subset \mathbb{R}^N \) is \( G \)-invariant, then the functional
        $$J_\lambda(u)=\frac{1}{2}\irn(|\nabla u|^2 + \lambda\mathds 1_\o u^2) - \frac{1}{2^*}\irn Q_\o(x)|u|^{2^*}$$
is $G$-invariant, i.e., $J_\lambda(gu)=J_\lambda(u)$ for every $g\in G$ and $u\in D^{1,2}(\rn)$. Therefore, by the principle of symmetric criticality \cite[Theorem 1.28]{w}, the critical points of the restriction of $J_\lambda$ to $D^{1,2}(\rn)^{G} $ are precisely the $G$-invariant solutions to the problem \eqref{eq:problem}. The nontrivial ones belong to the Nehari manifold
$$\cN_\lambda^G:=\{u\in D^{1,2}(\rn)^{G}:u\neq 0, \ J_\lambda(u)u=0\}.$$

\begin{lemma} \label{lem:symmetric_ps}
If $\#Gx=\infty$ for every $x\in\overline{\o}$, then the restriction of $J_\lambda$ to $D^{1,2}(\rn)^{G} $ satisfies the Palais-Smale condition, i.e., every sequence $(u_k)$ in $D^{1,2}(\rn)^{G}$ such that $J_\lambda(u_k)\to c$ and $J'_\lambda(u_k)\to 0$ in $(D^{1,2}(\rn))'$ contains a convergent subsequence.
\end{lemma}
        
\begin{proof}
Let $u_k\in D^{1,2}(\rn)^{G}$ satisfy $J_\lambda(u_k)\to c$ and $J'_\lambda(u_k)\to 0$ in $(D^{1,2}(\rn))'$. Arguing as in the proof of Theorem \ref{thm:struwe} we see that $(u_k)$ is bounded in $D^{1,2}(\rn)$ and that, after passing to a subsequence, $u_k \rh u$ weakly in $D^{1,2}(\rn)$, $v_k:=u_k-u\rh 0$ weakly in $D^{1,2}(\rn)$, $J_0(v_k)\to c_1$ and $J'_0(v_k)\to 0$ in $(D^{1,2}(\rn))'$. We claim that $u_k\to u$ strongly in $D^{1,2}(\rn)$.

Arguing by contradiction, assume that $(v_k)$ does not converge strongly to $0$. Then, by Lemma \ref{lem:struwe}, there are sequences $(\xi_k)$ in $\o$ and $(\eps_k)$ in $(0,\infty)$ such that $B_{\eps_k}(\xi_k)\subset\o$ and, passing to a subsequence, 
$$\int_{B_{\eps_k}(\xi_k)}|v_k|^{2^*}=\delta>0\qquad\text{for all \ }k\in\n.$$
Since $\o$ is bounded, a subsequence of $(\xi_k)$ converges to some point $\xi\in\overline{\o}$. As $\#G\xi=\infty$, for any given $n\in\n$ there exist $g_1,\ldots,g_n\in G$ such that $g_i\xi\neq g_j\xi$ if $i\neq j$. Set $3d:=\min_{i\neq j}|g_i\xi-g_j\xi|>0$ and choose $k_0\in \n$ so that $|\xi_k-\xi|<d$ if $k\geq k_0$.  Then, since every $g_i$ is an isometry,
$$3d\leq|g_i\xi-g_j\xi|\leq |g_i\xi-g_i\xi_k|+|g_i\xi_k-g_j\xi_k|+|g_j\xi_k-g_j\xi|<|g_i\xi_k-g_j\xi_k| + 2d\quad\text{if \ }k\geq k_0\text{ \ and \ }i\neq j.$$
Therefore, $|g_i\xi_k-g_j\xi_k|>d$ and, as a consequence, $B_{\eps_k}(g_i\xi_k)\cap B_{\eps_k}(g_j\xi_k)=\emptyset$ if $i\neq j$ and $k$ is large enough. Since $u_k$ is $G$-invariant, it follows that 
$$n\delta<n\int_{B_{\eps_k}(\xi_k)}|v_k|^{2^*}=\sum_{i=1}^n\int_{B_{\eps_k}(g_i\xi_k)}|v_k|^{2^*}\leq \io|v_k|^{2^*}\leq\irn |v_k|^{2^*}.$$
But $n$ was chosen arbitrarily. This contradicts $(v_k)$ being bounded in $D^{1,2}(\rn)$ and proves that $u_k\to u$ strongly in $D^{1,2}(\rn)$.
\end{proof}

\begin{proof}[Proof of Theorem \ref{thm:sym1}]
The Nehari manifold $\cN_\lambda^G$ is symmetric with respect to the origin, $0\notin\cN_\lambda^G$, and $J_\lambda$ is even and bounded from below on $\cN_\lambda^G$. By a standard argument, it follows from Lemma \ref{lem:symmetric_ps} that the restriction of $J_\lambda$ to $\cN_\lambda^G$ satisfies the Palais-Smale condition; see, e.g., \cite[Lemma 2.2]{chs}. 

Let $D^{1,2}_0(\o)^G$ be the space of $G$-invariant functions in $D^{1,2}_0(\o)$ and $\Sigma_\o^G:=\{u\in D^{1,2}_0(\o)^G:\|u\|_\lambda=1\}$ be the unit sphere in $D^{1,2}_0(\o)^G$ with the norm $\|\cdot\|_\lambda$. The map
$$\Sigma_\o^G\to\cN_\lambda^G,\qquad u\mapsto \left(\frac{\|u\|_\lambda^2}{\io |u|^{2^*}}\right)^\frac{1}{2^*-2}u,$$
is continuous and odd. Therefore, $\infty=\mathrm{genus}(\Sigma_\o^G)\leq\mathrm{genus}(\cN_\lambda^G)$, where ``$\mathrm{genus}$" means the Krasnoselskii genus.

It follows from \cite[Theorem II.5.7]{s} that $J_\lambda$ has infinitely many pairs of critical points on $\cN_\lambda^G$.
\end{proof}

\subsection*{Acknowledgments}

J. Faya would like to express his gratitude to M. Clapp and A. Saldaña from the Instituto de Matemáticas for their generous hospitality during his stay.

\subsection*{Disclosure statement}
The authors report there are no competing interests to declare.

\bigskip

\begin{flushleft}
\textbf{Mónica Clapp}\\
Instituto de Matemáticas\\
Universidad Nacional Autónoma de México \\
Campus Juriquilla\\
76230 Querétaro, Qro., Mexico\\
\texttt{monica.clapp@im.unam.mx} 
\medskip

\textbf{Jorge Faya}\\
Instituto de Ciencias Fisicas y Matem\'aticas\\
Universidad Austral de Chile\\
Facultad de Ciencias\\
Av. Rector Eduardo Morales Miranda 23\\
Valdivia, Chile\\
\texttt{jorge.faya@uach.cl}
\medskip

\textbf{Alberto Saldaña}\\
Instituto de Matemáticas\\
Universidad Nacional Autónoma de México \\
Circuito Exterior, Ciudad Universitaria\\
04510 Coyoacán, Ciudad de México, Mexico\\
\texttt{alberto.saldana@im.unam.mx}
\end{flushleft}
	
\end{document}